\documentclass[12pt]{article}

\usepackage{amsthm,amsmath,amssymb,amsrefs,blkarray}
\usepackage[colorlinks,citecolor=blue,urlcolor=blue]{hyperref}

\def\es{\emptyset}
\def\eps{\varepsilon}

\def\Z{\mathbb{Z}}
\def\R{\mathbb{R}}
\def\Av{\mathrm{Av}}

\def\dist{\mathrm{dist}}
\def\max{\mathrm{max}}
\def\op{\mathrm{open}}
\def\cl{\mathrm{closed}}
\def\UST{\mathsf{UST}}
\def\WSF{\mathsf{WSF}}
\def\cC{\mathcal{C}}

\def\cE{\mathcal{E}}

\def\cN{\mathcal{N}}
\def\cR{\mathcal{R}}

\def\cT{\mathcal{T}}

\def\LE{\mathrm{LE}}
\def\p{\mathbf{p}}
\def\Pr{\mathbf{P}}
\def\E{\mathbf{E}}
\def\one{\mathbf{1}}

\def\tail{\mathrm{tail}}
\def\head{\mathrm{head}}
\def\weight{\mathrm{weight}}
\def\outdeg{\mathrm{outdeg}}
\def\diff{\mathrm{diff}}
\def\UST{\mathsf{UST}}

\def\i1{\mathbf{i_1}}
\def\i1pr{\mathbf{i_1'}}
\def\j1{\mathbf{j_1}}
\def\j2{\mathbf{j_2}}
\def\j3{\mathbf{j_3}}
\def\j4{\mathbf{j_4}}
\def\jpr1{\mathbf{j''_1}}
\def\jpr2{\mathbf{j''_2}}
\def\jpr4{\mathbf{j''_4}}
\def\w1{\mathbf{w_1}}
\def\w3{\mathbf{w_3}}
\def\w4{\mathbf{w_4}}

\newcommand{\eqn}[2]{\begin{equation}\label{#1}#2\end{equation}}
\newcommand{\eqnst}[1]{\begin{equation*}#1\end{equation*}}
\newcommand{\eqnspl}[2]{\begin{equation}\begin{split}\label{#1}%
    #2\end{split}\end{equation}}
\newcommand{\eqnsplst}[1]{\begin{equation*}\begin{split}%
    #1\end{split}\end{equation*}}

\theoremstyle{plain}
\newtheorem{theorem}{Theorem}[section]
\newtheorem{lemma}[theorem]{Lemma}
\newtheorem{proposition}[theorem]{Proposition}
\newtheorem{corollary}[theorem]{Corollary}

\theoremstyle{definition}
\newtheorem{definition}[theorem]{Definition}
\newtheorem{exercise}[theorem]{Exercise}
\newtheorem{example}[theorem]{Example}

\theoremstyle{remark}
\newtheorem{remark}[theorem]{Remark}
\newtheorem{conjecture}[theorem]{Conjecture}
\newtheorem{open}[theorem]{Open Question}

\begin{document}

\title{Sandpile models}

\author{Antal A.~J\'arai \thanks{Department of Mathematical Sciences,
University of Bath, Claverton Down, Bath BA2 7AY, United Kingdom. 
Email: {\tt A.Jarai@bath.ac.uk}}}

\maketitle

\begin{abstract}
This survey is an extended version of lectures given at the Cornell Probability
Summer School 2013. The fundamental facts about the Abelian sandpile model 
on a finite graph and its connections to related models are presented.
We discuss exactly computable results via Majumdar and Dhar's method. 
The main ideas of Priezzhev's computation of the height probabilities in 2D
are also presented, including explicit error estimates involved in passing 
to the limit of the infinite lattice. We also discuss various questions 
arising on infinite graphs, such as convergence to a sandpile measure, and
stabilizability of infinite configurations. 
\end{abstract}

\textbf{AMS subject classification:} Primary 60K35; secondary 82B20 \\

\textbf{Key-words:} Abelian sandpile, chip-firing, uniform spanning tree, 
loop-erased random walk, Wilson's algorithm, burning bijection,
height probabilities.

\tableofcontents

\section{Introduction}
\label{sec:intro}

The Abelian sandpile model and close variants were introduced 
several times in different contexts independently. There is motivation
coming from statistical physics, probability and combinatorics. 
However, we are going to delay a detailed discussion of where the
model comes from to Section \ref{sec:motivation}, and start with its
definition and some of its basic properties in Section \ref{sec:model}. 
There are a number of reasons why this seems to be a good choice:
\begin{enumerate}
\item The basic model is very simple to define, and some of its
fundamental properties can be established without any serious 
pre-requisites. It is hoped that the model will have sufficient
appeal on its own without motivation in advance. 
\item We do not want to assume prior familiarity with statistical
physics models such as percolation, the Ising model, etc. However,
since the connection with critical phenomena is very important, 
it has to be explained, and it will be easier to do so when the 
basic model can be used as illustration. We have attempted to organize 
Section \ref{sec:motivation} in such a way that a reader unfamiliar
with statistical physics has a quick access to some important concepts.
\item As part of the motivation, we will also be ready to state some of
the main open questions. 
\end{enumerate}

There are a number of excellent surveys already available on 
sandpile models \cites{IP98,MRZ01,MRS05,Dhar06,Redig,HLMPPW,LP10,LP17}. 
Our focus is similar to that of Redig's 
notes \cite{Redig}, in that we cover rigorous results
roughly at the level of beginning PhD students. 
On the other hand, we have incorporated topics complementary to those in
\cite{Redig} and some results that are more recent. For example, we
discuss connections to the Tutte polynomial, the rotor-router walk
and a large part of Priezzhev's computation of height probabilities 
in 2D. Dhar's extensive survey \cite{Dhar06}, written from the point of view
of theoretical physics, will be an invaluable guide 
to anyone wanting to learn about the model.
An aspect of the theory that does not seem to receive much attention in 
the physics literature, though, is the precise arguments involving the 
limit of infinite graphs. Here we explain how this can be done based 
on one-endedness of components of the wired uniform spanning 
forest \cite{JW12}. The connection to the rotor-router model 
is due to \cite{HLMPPW}, that extends many of the basic results to 
directed graphs. For simplicity, in these notes we restricted attention 
to undirected graphs.

The outline of the paper is as follows. Section \ref{sec:model} 
introduces the sandpile Markov chain, recurrent configurations, 
the sandpile group and Dhar's formula for the average number of 
topplings. In Section \ref{sec:motivation} we give a 
brief introduction to critical phenomena using the percolation model
as example. Self-organized criticality is first illustrated with the 
forest-fire model built on the Erd\H{o}s-R\'{e}nyi random graph, that 
is perhaps the most intuitive example of the concept.
Then we discuss self-organized criticality in the Abelian sandpile 
model in terms of critical exponents. Section \ref{sec:connections} 
starts with the burning
bijection of Majumdar and Dhar and the connection to uniform spanning
trees. Following this we present the connections to the rotor-router 
model and the Tutte polynomial. Section \ref{sec:determinantal} 
is devoted to exactly computable results, and starts with Majumdar and 
Dhar's method. The scaling limit of the height $0$ field is discussed. 
Section \ref{ssec:height-123} is devoted to an exposition of the
computation of height probabilities in 2D due to Priezzhev, and
is followed by further 2D results in Section \ref{ssec:corr-123}.
Section \ref{sec:measures} is devoted to questions on infinite graphs
and highlights the role that properties of the wired uniform spanning 
forest play in infinite volume limits. Finally, 
Section \ref{sec:infinite-conf} discusses certain questions of 
stabilizability of infinite configurations.

\section{The Abelian sandpile model / chip-firing game on a finite graph}
\label{sec:model}

\subsection{Definition of the model}
\label{ssec:model-def}

Let $G = (V \cup \{ s \}, E)$ be a finite, connected multigraph
(i.e.~we allow multiple edges between vertices). The distinguished vertex
$s$ is called the \textbf{sink}. We exclude loop-edges for 
simplicity (their presence would involve only trivial modifications).
We write $\deg_G(x)$ for the degree of the vertex $x$ in the graph $G$,
and we write $x \sim y$ to denote that vertices $x$ and $y$ are
connected by at least one edge.

\begin{example}
\label{example:wired-graph}
Let $V \subset \Z^d$ be finite. Identify all vertices in $V^c = \Z^d \setminus V$
into a single vertex that becomes the sink $s$. Then remove all loop-edges
at $s$. This is called the \textbf{wired graph} induced by $V$.
Instead of $\Z^d$, we can start from any locally finite, infinite,
connected graph.
\end{example}

A \textbf{sandpile} is a collection 
of indistinguishable particles (chips, sand grains, etc.) on the 
vertices in $V$. A sandpile is hence specified by a map
$\eta : V \to \{ 0, 1, 2, \dots \}$. We say that $\eta$ is 
\textbf{stable at $x \in V$}, if $\eta(x) < \deg_G(x)$, and 
we say that it is \textbf{stable}, if it is stable at all 
$x \in V$.

We now introduce a dynamics that stabilizes any unstable sandpile.
If $\eta$ is unstable at some $x \in V$ (i.e.~$\eta(x) \ge \deg_G(x)$),
$x$ is \textbf{allowed to topple} which means that $x$ sends one
particle along each edge incident to it. (In the combinatorics literature
it is common to say \textbf{the vertex $x$ fires} by sending chips
to its neighbours.) On toppling the vertex $x$, the particles are 
re-distributed as follows:
\eqnsplst
{ \eta(x) &\ \longrightarrow\ \eta(x) - \deg_G(x) \\
  \eta(y) &\ \longrightarrow\ \eta(y) + a_{xy}, \quad y \in V,\, y \not= x, }
where $a_{xy} = \text{number of edges between $x$ and $y$}$.
Regarding $\eta$ as a row vector, this can be concisely written as
\eqn{e:topple}
{ \eta \longrightarrow \eta - \Delta'_{x,\cdot}, }
where 
\eqnsplst
{ \Delta'_{xy}
  &= \begin{cases}
     \deg_G(x) & \text{if $x = y \in V$;}\\
     -a_{xy}   & \text{if $x \not= y$, $x, y \in V$;}
     \end{cases} \\
  \Delta'_{x,\cdot}
  &= \text{row $x$ of $\Delta'$}. }
That is, if $\Delta = \text{graph Laplacian of $G$}$ 
then $\Delta' = \text{restriction of $\Delta$ to $V \times V$}$.
Particles arriving at the sink are lost, that is, we do not
keep track of them. Observe that requiring $\eta(x) \ge \deg_G(x)$
before toppling ensures that we still have a sandpile 
after toppling (i.e.~the number of particles at $x$ is still 
non-negative after toppling). We also say in this case that toppling $x$ is
\textbf{legal}.

Toppling a vertex may create further unstable 
vertices. 

\begin{definition}
Given a sandpile $\xi$, we define its \textbf{stabilization}
\eqnst
{ \xi^\circ \in \Omega_G 
  := \{ \text{stable sandpiles} \}
  = \prod_{x \in V} \{ 0, 1, \dots, \deg_G(x) - 1 \}, }  
by carrying out all possible legal topplings, in any order, until
a stable sandpile is reached.
\end{definition}

\begin{theorem}
\label{thm:stabilization}
\cite{Dhar90}
The map $\xi \mapsto \xi^\circ$ is well-defined.
\end{theorem}

\begin{proof}
We need to show:
\begin{itemize}
\item[(a)] Only finitely many topplings can occur, regardless 
of how we choose to topple vertices.
\item[(b)] The final stable configuration is independent of the
sequence of topplings chosen.
\end{itemize}
In order to see (a), observe that if $x \sim s$ then $x$ can
topple only finitely many times (the system loses particles 
to $s$ on each toppling of $x$). It follows by induction that 
for all $k \ge 1$, if $x \sim x_{k-1} \sim \dots \sim x_1 \sim s$,
then $x$ can topple only finitely many times. Since $G$
is connected, we are done.

We now prove (b) in two steps.

(i) \emph{``Topplings commute''}. If $x, y \in V$, $x \not= y$ and
$\eta$ is unstable at both $x$ and $y$, then writing $T_x$ to denote
the effect of toppling $x$ we claim that 
\eqn{e:topplings-commute}
{ T_y T_x \eta 
  = T_x T_y \eta. }
Observe that in either order, both topplings are legal.
Then the claim is immediate from observing that both sides equal
$\eta - \Delta'_{x,\cdot} - \Delta'_{y,\cdot}$.

(ii) Suppose now that 
\eqn{e:1st}
{ x_1, x_2, \dots, x_k }
and 
\eqn{e:2nd}
{ y_1, y_2, \dots, y_\ell }
are two sequences of vertices that are both possible stabilizing sequences 
of $\eta$. That is, when carried out in order from left to right,
in both sequences each toppling is legal, and the final results
are stable configurations. If $\eta$ is already stable, then 
$k = \ell = 0$ and there is nothing to prove.

Otherwise, we have $k, \ell \ge 1$ and $\eta(x_1) \ge \deg_G(x_1)$.
Therefore, $x_1$ must occur somewhere in the second sequence, otherwise
the second sequence would never reduce the number of particles at $x_1$.
Let $x_1 = y_i$, $1 \le i \le \ell$, and suppose that $i$ is the smallest 
such index. By part (i), the toppling of $y_i = x_1$ can be moved 
to the front of the second stabilizing sequence. Precisely, we have
\eqnsplst
{ T_{x_1} T_{y_{i-1}} \dots T_{y_1} \eta
  &= T_{y_{i-1}} T_{x_1} T_{y_{i-2}} \dots T_{y_1} \eta \\
  &= T_{y_{i-1}} T_{y_{i-2}} T_{x_1} \dots T_{y_1} \eta \\
  &\ \ \vdots \\
  &= T_{y_{i-1}} T_{y_{i-2}} \dots T_{y_1} T_{x_1} \eta. }
It follows that the sequence
\eqn{e:2nd'}
{ x_1, y_1, \dots, y_{i-1}, y_{i+1}, \dots, y_\ell }
also stablizes $\eta$. We now remove $x_1$ from the beginning 
of the sequences \eqref{e:1st} and \eqref{e:2nd'} and repeat the 
argument for $T_{x_1} \eta$. Iterating gives that $k = \ell$ and the 
multisets $[x_1,\dots,x_k]$ and $[y_1,\dots,y_\ell]$ are permutations 
of each other. That is, each vertex topples the same number of times in 
the two stabilizing sequences, and hence they reach the same final 
configuration. This completes the proof that the
stabilization $\xi \mapsto \xi^\circ$ is well-defined.
\end{proof}

\begin{remark}
Sometimes, especially in the physics literature, a stable sandpile is 
defined as having possible values $1, \dots, \deg_G(x)$ at $x$, and
a toppling of $x$ is allowed when $\eta(x) > \deg_G(x)$.
It is easy to see that this merely amounts to a shift
of coordinates, and defines the same model.
\end{remark}

\medbreak

{\bf Motivating remarks.} The sandpile dynamics can be
viewed as a toy model of avalanche-type phenomena. Adding a single particle 
to the pile and stabilizing can induce a complex sequence of topplings,
called an ``avalanche''. However, the model is \emph{not intended} as a 
realistic model of sand. In order to model sand grains moving down a slope, 
a more suitable condition for toppling could be that the \emph{discrete gradient} 
exceeds some fixed critical value $d_c > 0$. It is easy to see that in such 
models topplings do not commute. In fact, if $y_1 \sim x \sim y_2$ and
$\eta(x) - \eta(y_1) = d_c = \eta(x) - \eta(y_2)$, one
needs to make a choice in the model if a particle from $x$ will move to 
$y_1$ or $y_2$. We will see later that commutativity in the
Abelian sandpile has many nice consequences, which make
it more amenable to study. The point is that the Abelian model 
already possesses important qualitative features of avalanche-like 
phenomena and as we will see, has very nontrivial behaviour.
We will return to this in Section \ref{sec:motivation}.

\medbreak

\begin{exercise} \emph{(Asymmetric sandpile model)}
Let $G = (V \cup \{ s \}, E)$ be a \emph{directed} graph with a distinguished vertex $s$.
Find appropriate definitions of ``stable'' and ``toppling''.
Find a condition on $G$ that ensures that stabilization is well-defined.
See \cite{HLMPPW}.
\end{exercise}

\begin{exercise} \emph{(Least action principle)}
Check that the argument of Theorem \ref{thm:stabilization}(b) gives the 
following stronger statement. 
Suppose that $x_1, x_2, \dots, x_k$ is a stabilizing
sequence for $\eta$ consisting of legal topplings. Suppose that 
$y_1, y_2, \dots, y_\ell$ is any other sequence of possibly illegal 
topplings, such that carrying them out results in a stable configuration. 
Then each vertex is toppled at least as many times in the $y$-sequence 
as in the $x$-sequence. In other words, with legal topplings, 
each vertex does the minimum amount of required ``work'' to
stabilize the configuration. See \cite{FLP10} for more on 
this ``least action principle''.
\end{exercise}

\begin{definition}
The \textbf{addition operators} are the maps on sandpiles 
defined by adding one particle at $x$ and stabilizing.
More formally, $E_x \eta = ( \eta + \one_x )^\circ$, where
$\one_x$ is the row vector with $1$ in position $x$ and $0$ 
elsewhere. The sequence of topplings carried out in stabilizing 
$\eta + \one_x$ is called the \textbf{avalanche} induced by
the addition.
\end{definition}

\begin{lemma}
\label{lem:Abelian}
\cite{Dhar90}
We have $E_x E_y = E_y E_x$ for all $x, y \in V$.
\end{lemma}

\begin{proof}
We have
\eqn{e:ExEy}
{ E_x E_y \eta 
  = ( ( \eta + \one_y )^\circ + \one_x )^\circ }
and 
\eqn{e:EyEx}
{ E_y E_x \eta 
  = ( ( \eta + \one_x )^\circ + \one_y )^\circ. }
We show that both expressions equal
\eqn{e:both}
{ ( \eta + \one_x + \one_y )^\circ. }
To see this, start with the configuration $\eta + \one_x + \one_y$, 
and carry out topplings as in the stabilization of $\eta + \one_y$. 
The extra particle present at $x$ does not affect the legality
of any of the topplings. Hence with the extra particle at $x$ present, 
we arrive at the configuration $( \eta + \one_y )^\circ + \one_x$.
Now carry out any further topplings that are possible, arriving at
the right hand side of \eqref{e:ExEy}. Due to Theorem \ref{thm:stabilization},
the final configuration also equals \eqref{e:both}. Equality of
\eqref{e:both} and \eqref{e:EyEx} is seen similarly.
\end{proof}

So far the dynamics has been determinisitic. We now add randomness
and define the \textbf{sandpile Markov chain} as follows. We take as
state space the set $\Omega_G$ of stable sandpiles. Fix a positive 
probability distribution $p$ on $V$, i.e.~$\sum_{x \in V} p(x) = 1$ 
and $p(x) > 0$ for all $x \in V$. Given the current state
$\eta \in \Omega_G$, pick a random vertex $X \in V$ according to $p$,
add a particle there, and stabilize to obtain the next state of the
Markov chain. That is, the Markov chain makes the transition
\eqnst
{ \eta \ \longrightarrow\  E_X \eta = ( \eta + \one_X )^\circ. }
If the Markov chain has initial state $\eta_0$, we can write 
the time evolution, using Theorem \ref{thm:stabilization} and
Lemma \ref{lem:Abelian}, as
\eqnst
{ \eta_n
  = \left( \eta_0 + \sum_{i=1}^n \one_{X_i} \right)^\circ
  = E_{X_n} \dots E_{X_1} \eta_0, }
where $X_1, X_2, \dots$ are i.i.d.~random variables distributed
according to $p$.

We denote by $\eta^\max$ the sandpile defined by 
$\eta^\max(x) = \deg_G(x) - 1$, $x \in V$. For sandpiles $\eta, \xi$, 
we write $\eta \ge \xi$, if $\eta(x) \ge \xi(x)$ for all $x \in V$.

Recall the following standard terminology for general Markov chains.
For two states $\eta, \xi$ of a Markov chain we say that
\emph{$\xi$ can be reached from $\eta$}, if there exists $n \ge 0$ such that
$\mathbf{p}^n(\eta,\xi) > 0$, where $\mathbf{p}^n$ is the $n$-step
transition probability. We say that \emph{$\eta$ and $\xi$ communicate},
if they can be reached from each other. This is an equivalence relation,
and the equivalence classes are the \emph{communicating classes}.
A state $\eta$ is \emph{recurrent}, if starting from $\eta$ the chain returns
to $\eta$ with probability $1$, and it is \emph{transient} otherwise. 
Recurrence and transience are \emph{class properties}, that is, if
one state in a class is recurrent then all states are.

\begin{theorem}
\label{thm:recurrent}
\cites{Dhar90,HLMPPW}
Consider the sandpile Markov chain on any finite connected multigraph
$G = (V \cup \{ s \}, E)$ (satisfying $p(x) > 0$ for all $x \in V$).\\
(i) There is a single recurrent class.\\
(ii) The following are equivalent for $\eta \in \Omega_G$:
\begin{itemize}
\item[(a)] $\eta$ is recurrent;
\item[(b)] there exists a sandpile $\xi \ge \eta^\max$ such that 
$\xi^\circ = \eta$;
\item[(c)] for any sandpile $\sigma$, it is possible to reach $\eta$ from 
$\sigma$ by adding particles and toppling vertices, i.e.~there exists a 
sandpile $\zeta$ such that $\eta = (\sigma + \zeta)^\circ$.
\end{itemize}
\end{theorem}

\begin{proof}
(i) The configuration $\eta^\max$ is reachable for the Markov chain from
any $\zeta \in \Omega_G$ (by addition of particles). Hence $\eta^\max$ is
recurrent, and the recurrent class containing it is the only recurrent class.

(ii) (a) $\Longrightarrow$ (b). If $\eta$ is recurrent, it is reachable
from $\eta^\max$, i.e.~there exist $k \ge 0$ and $x_1, \dots, x_k \in V$ such 
that 
\eqnst
{ \eta 
  = E_{x_k} \dots E_{x_1} \eta^\max 
  = \left( \eta^\max + \sum_{i = 1}^k \one_{x_i} \right)^\circ. }
Hence we can take $\xi$ to be the configuration inside the parentheses
on the right hand side.

(b) $\Longrightarrow$ (a). If $\xi^\circ = \eta$, $\xi \ge \eta^\max$,
we can write $\xi = \eta^\max + \sum_{i=1}^k \one_{x_i}$ with 
some $x_1, \dots, x_k \in V$, so that $\eta = E_{x_k} \dots E_{x_1} \eta^\max$.
This shows that $\eta$ is reachable from $\eta^\max$, and hence 
recurrent.

(c) $\Longrightarrow$ (b). This is obvious by taking $\sigma = \eta^\max$.

(b) $\Longrightarrow$ (c). 
Let $\xi \ge \eta^\max$ be such that 
$\xi^\circ = \eta$. Take $\zeta := \xi - \sigma^\circ \ge \eta^\max - \sigma^\circ \ge 0$. 
Then since $\xi - \sigma^\circ \ge 0$, starting from 
$\sigma + \zeta = \sigma + \xi - \sigma^\circ$ we can legally topple
a sequence of vertices that stabilizes $\sigma$, and arrive at the
configuration $\sigma^\circ + \xi - \sigma^\circ = \xi$. 
Now we can legally topple a sequence of vertices that stabilizes $\xi$, 
and arrive at the configuration $\eta$.
This shows that $\eta$ has property (c).
\end{proof}

\begin{definition}
We denote by $\cR_G$ the set of recurrent sandpiles. 
\end{definition}

\subsection{The sandpile group / critical group}
\label{ssec:sandpile-group}

Let $G = (V \cup \{ s \}, E)$ be a finite connected multigraph.
We now define the sandpile group of $G$.
Consider $\Z^V$ as an Abelian group. The integer row span
$\Z^V \Delta'_G$ of the matrix $\Delta'_G$ forms a subgroup of
$\Z^V$. For $\xi, \zeta \in \Z^V$, let us write $\xi \sim \zeta$
if $\xi - \zeta \in \Z^V \Delta'_G$. This is an equivalence
relation, and we write $[\xi]$ to denote the equivalence
class containing $\xi$. The equivalence classes form 
an Abelian group, the factor group
\eqnst
{ K_G 
  := \Z^V / \Z^V \Delta'_G. }
The group $K_G$ is called the \textbf{sandpile group} of $G$
(sometimes called the \textbf{critical group} in the 
combinatorics literature). Any toppling corresponds to
subtracting a row of $\Delta'_G$ from the configuration
(recall \eqref{e:topple}). Therefore, during
stabilization a configuration is replaced by an equivalent one.
Hence we can expect that the group $K_G$ plays a role in
understanding the sandpile Markov chain. This is made
precise by the following theorem.

\begin{theorem}
\label{thm:group}
\cite{Dhar90}
(i) Every equivalence class in $K_G$ contains precisely one recurrent 
sandpile in $\cR_G$. In particular, 
\eqnst
{ | \cR_G |
  = | K_G |
  = \det(\Delta'_G). }
(ii) Consequently, the following operation 
$\oplus : \cR_G \times \cR_G \to \cR_G$ turns $\cR_G$ into an Abelian
group isomorphic to $K_G$:
\eqnst
{ \eta \oplus \xi
  := (\eta + \xi)^\circ. }
\end{theorem}

The proof of (i) we give is due to \cite{HLMPPW}.
We will need the following lemma of \cite{HLMPPW} that provides a
configuration with a special property (later it will become clear that
this is a representative of the identity of $K_G$).

\begin{lemma}
\label{lem:id-like}
\cite{HLMPPW}
Let $\eps := \delta - \delta^\circ$, where $\delta$ is defined 
by $\delta(x) = \deg_G(x)$, $x \in V$. If $\eta$ is recurrent, then
$(\eta + \eps)^\circ = \eta$.
\end{lemma}

\begin{proof}
By Theorem \ref{thm:recurrent}, if $\eta$ is recurrent, it is possible to
add particles to $\delta$ and stabilize to get $\eta$. That is, there
exists $\zeta \ge 0$ such that $\eta = (\delta + \zeta)^\circ$. Consider the
configuration
\eqnst
{ \gamma
  = (\zeta + \delta) + \eps
  = \delta + \zeta + \delta - \delta^\circ. }
Since $\eps \ge 0$, we can start from $\gamma$ and legally topple a sequence
of vertices that stabilizes $\zeta + \delta$, arriving at the 
configuration $\eta + \eps$. Stabilizing further gives $(\eta + \eps)^\circ$. 
On the other hand, since $\delta - \delta^\circ \ge 0$,
we can start from $\gamma$, and legally topple a sequence of vertices 
that stabilizes $\delta$, arriving at the configuration 
$\delta^\circ + \zeta + \delta - \delta^\circ = \zeta + \delta$. 
Stabilizing further we obtain $\eta$. Comparing the two stabilizing 
sequences, Theorem \ref{thm:stabilization}(ii) yields
$\eta = (\eta + \eps)^\circ$.
\end{proof}

\begin{proof}[Proof of Theorem \ref{thm:group}.]
(i) We first observe that every equivalence class contains some
representative $\xi$ with $\xi \ge \eta^\max$. Then
$\xi^\circ =: \eta \in \cR_G$ by Theorem \ref{thm:recurrent}(ii), and
$\eta \in [\xi]$. It follows that $\cR_G$ intersects each 
equivalence class.

It remains to show that the intersection of $\cR_G$ with any equivalence
class contains at most one element. To see this, suppose that 
$\eta_1 \sim \eta_2$, $\eta_1, \eta_2 \in \cR_G$, and we show $\eta_1 = \eta_2$.
Since $\eta_1 \sim \eta_2$, there exist $c_x \in \Z$, $x \in V$, such that 
\eqnst
{ \eta_1
  = \eta_2 + \sum_{x \in V} c_x \Delta'_{x,\cdot}. }
Let $V_- := \{ x \in V : c_x < 0 \}$ and $V_+ := \{ x \in V : c_x > 0 \}$, and define
\eqnst
{ \eta 
  := \eta_1 + \sum_{x \in V_-} (-c_x) \Delta'_{x,\cdot}
  = \eta_2 + \sum_{x \in V_+} c_x \Delta'_{x,\cdot}. }  
Take $k$ large enough such that the configuration $\eta'$ defined by
$\eta' = \eta + k \eps$ satisfies $\eta'(x) \ge |c_x| \deg_G(x)$ for all $x \in V$.
This is possible, since each entry of $\eps$ is at least $1$.
Starting from $\eta'$, we may legally topple $(-c_x)$-times 
each vertex $x \in V_-$, arriving at the configuration
$\eta_1 + k \eps$. This further stabilizes to $\eta_1$, by Lemma \ref{lem:id-like}.
Similarly, we can legally topple $c_x$-times each vertex
$x \in V_+$, arriving at the configuration $\eta_2 + k \eps$. 
This further stabilizes to $\eta_2$, again by Lemma \ref{lem:id-like}.
Comparing the two stabilzations, Theorem \ref{thm:stabilization}(ii)
yields $\eta_1 = \eta_2$.

The number of elements of $K_G$ is the index of the subgroup $\Z^V \Delta'_G$.
It is easy to see that this equals the determinant of $\Delta'_G$.

(ii) It is not difficult to show (see Exercise \ref{ex:oplus}) that 
if $\eta, \xi \in \cR_G$, we have $(\eta + \xi)^\circ \in \cR_G$.
Hence $\oplus$ indeed maps $\cR_G \times \cR_G$ into $\cR_G$.
We also have 
\eqnst
{ [\eta \oplus \xi]
  = [ (\eta + \xi)^\circ ]
  = [ \eta + \xi ]
  = [ \eta ] + [ \xi ]. }
This shows that $\oplus$ indeed corresponds to the group operation in $K_G$.
\end{proof}

\begin{exercise}
\label{ex:oplus}
Show that if $\eta, \xi \in \cR_G$, then $(\eta + \xi)^\circ \in \cR_G$.
(\emph{Hint:} One way to see this is the criterion of 
Theorem \ref{thm:recurrent}(ii)(b).)
See \cite{Dhar90} and \cite{HLMPPW}*{Corollary 2.16}.
\end{exercise}

The identity element of $K_G$, that is, the unique recurrent 
configuration $\eta_0 \in \cR_G$ such that $[\eta_0] = [0]$, 
displays highly non-trivial features. On large rectangular regions
in $\Z^2$, the identity element displays both regular and fractal 
patterns; see the pictures in \cites{BR02,LP10}. 
Le Borgne and Rossin \cite{BR02} prove some rigorous results
for rectangular regions.

\begin{exercise}
\label{ex:sink-nbr}
Show that if there is a unique $x \in V$ such that $x \sim s$, then
$E_x$ restricted to $\cR_G$ is the identity. 
(\emph{Hint:} Show that ${\bf 1}_x$ equals the sum of the rows
of $\Delta'_G$. This is a special case of Dhar's ``multiplication
by identity test''; see \cite{Dhar90}.)
\end{exercise}

\begin{exercise}
\label{ex:smaller-gen}
Suppose that $V \subset \Z^d$, and $G$ is the wired graph induced by $V$.
Show that if $p(x) > 0$ for all $x \in V$ such that $x \sim s$, then
the sandpile Markov chain is irreducible. (That is, in this case the 
condition imposed on $p$ in Theorem \ref{thm:recurrent} can be substantially
relaxed.) See \cite{MRZ01}.
\end{exercise}

\subsection{The stationary distribution and Dhar's formula}
\label{ssec:stationary}

Once the sandpile Markov chain reaches the set $\cR_G$ of 
recurrent sandpiles, it never leaves it. Let us write
$\nu_G$ for the stationary distribution, which by 
Theorem \ref{thm:recurrent}(i) is unique, and is concentrated
on $\cR_G$. In view of Theorem \ref{thm:group}, 
the restriction of the Markov chain to $\cR_G$ can be identified 
with a random walk on the finite group $K_G$. A transition from 
the state $\eta \in \cR_G$ to $(\eta + \mathbf{1}_x)^\circ$, $x \in V$,
is identified with adding to $[\eta] \in K_G$ the 
group element $[\mathbf{1}_x]$. As the next easy exercise shows, this 
implies that the stationary distribution of the sandpile Markov chain is
\emph{uniform} on $\cR_G$.

\begin{exercise}
\label{ex:rw-group}
Let $K$ be a finite group, and let $(X_n)_{n \ge 0}$ be an
irreducible random walk on $K$, that is a Markov chain with transition 
matrix $P(h,gh) = \mu(g)$, where $\sum_{g \in K} \mu(g) = 1$, and 
the support of $\mu$ generates $K$. Show that the stationary 
distribution of $(X_n)_{n \ge 0}$ is uniform on $K$. 
See \cite{S-C}.
\end{exercise}

The following exercise is essentially a triviality. However, we prefer 
to state it explicitly for two reasons: (i) it will play an important role in 
Theorem \ref{thm:Dhar's-formula} below; (ii) its version for infinite
graphs is far from trivial.

\begin{exercise}
\label{ex:inv-addition}
Check that the additition operators $E_x : \cR_G \to \cR_G$, $x \in V$,
leave the measure $\nu_G$ invariant. See \cites{Dhar90,MRZ01}.
\end{exercise}

The following theorem due to Dhar \cite{Dhar90} gives a formula 
for the average number of topplings induced by adding a single particle, 
under stationarity. 
For a sandpile $\eta$ and $x, y \in V$, let us write $n(x,y;\eta)$ for 
the number of topplings occurring at $y$ during the stabilization 
of $\eta + \one_x$. 

\begin{theorem}
\label{thm:Dhar's-formula}
Let $G = (V \cup \{ s \},E)$ be a finite connected multigraph. 
We have
\eqn{e:Dhar's-formula}
{ \E_{\nu_G} [ n(x,y,\cdot) ]
  = (\Delta'_G)^{-1}_{xy}, \quad x, y \in V. }
\end{theorem}

\begin{proof}
From the definition of stabilization we have the relation:
\eqnst
{ (\eta + \one_x)^\circ(y)
  = \eta(y) + \one_x(y) - \sum_{z \in V} n(x,z;\eta) (\Delta'_G)_{zy}. }
Now average both sides with respect to the stationary distribution
$\nu_G$. We get
\eqnst
{ \E_{\nu_G} [ (\eta + \one_x)^\circ(y) ] 
  = \E_{\nu_G} [ \eta(y) ] + \one_x(y)
    - \sum_{z \in V} \E_{\nu_G} [ n(x,z;\eta) ] (\Delta'_G)_{zy}. }
Due to Exercise \ref{ex:inv-addition}, the left hand side equals
the first term on the right hand side, which gives
\eqnst
{ \sum_{z \in V} \E_{\nu_G} [ n(x,z;\eta) ] (\Delta'_G)_{zy}
  = \one_x(y). }
Since this holds for all $x, y \in V$, we get \eqref{e:Dhar's-formula}.
\end{proof}

The above theorem is extremely useful in estimating topplings in an 
ava\-lanche. However, it only gives information about the first moment
of the toppling numbers.

\begin{open}
Find a useful expression for the second moment of the toppling numbers.
\end{open}

\section{Motivation from statistical physics}
\label{sec:motivation}

In the physics literature, the sandpile model appears in connection 
with the notion of ``self-organized criticality'' (SOC) \cites{Dhar90,Dhar06}. 
In order to explain what SOC is, we first clarify the meaning 
of ``criticality'' in the example of percolation in Section \ref{ssec:percolation}.
In Section \ref{ssec:SOC} the notion of SOC is illustrated via an example
closely related to percolation. In Section \ref{ssec:SOC-sandpile}, we discuss 
SOC in the sandpile model and state some open problems.
In Section \ref{sec:connections} various connections to other models 
will be presented as well.

\subsection{Percolation --- an example of a critical phenomenon}
\label{ssec:percolation}

A simple-to-define but deep example of criticality is provided by 
\textbf{bond percolation} on the $d$-dimensional integer lattice $\Z^d$. 
Let $0 < p < 1$. Declare each edge of $\Z^d$ (also called a bond)
\textbf{occupied} with probability $p$ and \textbf{vacant} with 
probability $1-p$, independently. Percolation theory studies the
geometry of the connected components (called \textbf{clusters})
of the random subgraph of $\Z^d$ induced by the occupied bonds. 
We are going to write $\Pr_p$ for the underlying probability measure
when the parameter value is $p$. 
Let $\cC$ denote the connected occupied component containing the origin. 
A fundamental result
in percolation theory is the following theorem due to Broadbent and
Hammersley \cites{BH57,Hamm57a,Hamm59}.

\begin{theorem}
\label{thm:phase-transition}
Let $d \ge 2$. There exists a critical probability $0 < p_c = p_c(d) < 1$ 
such that 
\eqnsplst
{ \Pr_p [ |\cC| < \infty ] = 1, \quad \text{ if $p < p_c$;}\\
  \Pr_p [ |\cC| = \infty ] > 0, \quad \text{ if $p > p_c$.} }
\end{theorem}

It is easy to see using translation invariance that in the 
case $p < p_c$ there is no infinite cluster
anywhere in the lattice $\Pr_p$-a.s. It can also be shown that in the case
$p > p_c$ there exists a \emph{unique} infinite cluster \emph{somewhere}
in the lattice; see \cite{Grimmett}.
Note the qualitative similarity with extinction/survival for a 
branching process depending on whether the mean offspring is
less than or greater than $1$. One says that a \textbf{phase transition} occurs
at the \textbf{critical value} $p = p_c$ as the parameter $p$ is 
increased. The critical value separates the \textbf{subcritical phase}
$p < p_c$ where all clusters are finite a.s., and 
the \textbf{supercritical phase} $p > p_c$ where there exists an
infinite cluster a.s.


Percolation at the critical point $p = p_c$ has features that set it
apart from the sub- and supercritical phases. 
For example, it is known that for percolation with $p \not= p_c$, the probabilities
$\Pr_p [ |\cC| = k ]$ decay fast with $k$. The results \cite{Grimmett}*{Theorem 6.78},
\cite{Grimmett}*{Theorem 8.65} say that when $d \ge 2$, we have
\eqnsplst
{ \Pr_p \Big[ |\cC| \ge k \Big] 
  &\le C_1(p) \exp ( - c_1(p) k ) \quad \text{ when $p < p_c$;}\\ 
  \Pr_p \Big[ k \le |\cC| < \infty \Big] 
  &\le C_2(p) \exp \big( - c_2(p) k^{(d-1)/d} \big) \quad \text{ when $p > p_c$;}\\
  \Pr_p \Big[ x\in \cC,\, |\cC| < \infty \Big] 
  &\le \exp \big( - c_3(p) |x| \big) \quad \text{ when $p \not= p_c$.} } 
While these results are already not easy to establish, the case of $p = p_c$ is 
even more challenging. For example, it is a major conjecture  
that $\Pr_{p_c} [ |\cC| < \infty ] = 1$ in all dimensions $d \ge 2$.
So far, this has only been established in the planar case $d = 2$ \cites{Harris,Kesten80}
and when $d$ is sufficiently large ($d \ge 15$) \cites{BA91,HS90,vdHF,F13}.
A more detailed conjecture is that in all dimensions $d \ge 2$, the behaviour at and 
close to $p = p_c$ is characterized by power laws. For example, it
is expected that there exist \textbf{critical exponents} 
$\beta, \delta, \rho, \gamma, \eta \ge 0$, depending on the dimension $d$, 
such that 
\eqnspl{e:exponents}
{ \Pr_p [ |\cC| = \infty ]
  &= (p - p_c)^{\beta + o(1)} \quad \text{ as $p \downarrow p_c$;} \\
  \Pr_{p_c} [ |\cC| \ge k ]
  &= k^{-1/\delta + o(1)} \quad \text{ as $k \to \infty$;} \\
  \Pr_{p_c} [ \mathrm{rad}(\cC) \ge n ] 
  &= n^{-1/\rho + o(1)} \quad \text{ as $n \to \infty$;} \\
  \E_p [ |\cC|; |\cC| < \infty ] 
  &= |p - p_c|^{-\gamma + o(1)} \quad \text{ as $p \to p_c$;} \\
  \Pr_{p_c} [ x \in \cC ]
  &= \frac{1}{|x|^{d - 2 + \eta + o(1)}} \quad \text{ as $|x| \to \infty$;} }
where $\mathrm{rad}(\cC) = \sup \{ |x| : x \in \cC \}$. 
A further conjecture of \textbf{universality} states that the values
of the exponents are not sensitive to the structure of the lattice. 
In particular, they would not change if the cubic lattice is replaced
by any other $d$-dimensional periodic lattice. 

Most progress on the conjectures \eqref{e:exponents} has been made in
$d = 2$ and in high dimensions. In the planar case,
replace bond percolation on $\Z^2$ by so-called \textbf{site percolation}
on the triangular lattice. Here the \emph{vertices} of the triangular lattice
are declared occupied/vacant with probabilities $p$ and $1-p$, 
and the nearest neighbour occupied connected components are considered.
The combined result of the papers \cites{Kesten87,Smirnov01,LSW02,SW01}
is that for site percolation on the triangular grid we have
$\beta = 5/36$, $\delta = 91/5$, $\rho = 48/5$, $\gamma = 43/18$, $\eta = 5/24$.
In sufficiently high dimensions ($d \ge 11$) it has been established
that $\beta = 1, \delta = 2$, $\rho = 1/2$, $\gamma = 1$, $\eta = 0$;
\cites{BA91,HS90,HS00a,HS00b,HvdHS03,KN11,vdHF,F13}. It is conjectured that 
these are the values of the exponents for all $d > 6$. 
(The exponents have been established for all $d > 6$ in a modified model 
where long bonds are allowed; see the above references.)

\subsection{Self-organized criticality}
\label{ssec:SOC}

At first it seems that the intriguing properties of critical percolation
are very sensitive to the fact that we are at the critical point: $p = p_c$ 
exactly, or (in large finite systems) $p \approx p_c$. However,
there are interesting examples of \emph{dynamically evolving} models
where criticality (i.e.~power law behaviour of various distributions)
occurs in a rather robust fashion.

An example can be built on top of the Erd{\H{o}}s-R{\'e}nyi random graph 
model \cite{Bbook}, that can be viewed as percolation on the complete graph 
with $n$ vertices. It will be useful to consider the random graph in a
dynamical fashion. At time $t=0$ we have $n$ vertices and no edges.
Between any pair of vertices, independently, an edge is added at rate $1/n$.
This choice of rates ensures that locally around each vertex $O(1)$ edges 
appear per unit time. Write 
\eqnst
{ v^n_k(t)
  = \text{proportion of vertices in clusters of size $k$ at time $t$}, }
where the `size' of a cluster is the number of its vertices. 
In the limit $n \to \infty$ there is a phase transition. 
A giant component containing a positive fraction of all vertices
emerges at the critical time $t_c = 1$. One way to formalize
this statement is that $v^n_k(t)$ converges in probability 
to a deterministic limit $v_k(t)$, as $n \to \infty$, where the 
limit satisfies:
\eqnst
{ \theta(t)
  := 1 - \sum_{k \ge 1} v_k(t)
  \begin{cases}
  = 0 & \text{if $t \le 1$;} \\
  > 0 & \text{if $t > 1$.}
  \end{cases} }
Compare with Theorem \ref{thm:phase-transition}. At the critical time 
$t_c = 1$, we have the power law $v_k(t_c) \sim c k^{-3/2}$, 
as $k \to \infty$.

\medbreak

\emph{Forest-fire model.}
Let us modify the dynamics in a way that prevents the giant component
from emerging, and keeps the system ``at criticality''. Return to the
finite $n$ model, and let $\lambda(n)$ be a function satisfying
$1/n \ll \lambda(n) \ll 1$.
Suppose that ``lightning'' hits each vertex, independently, at rate
$\lambda(n)$. When a vertex is hit by lightning, the cluster containing
it disintegrates into individual vertices, that is, all edges within
the cluster become vacant instantaneously. Heuristically, this mechanism 
should prevent clusters to reach size of order $n$, since then it is
extremely likely that they are hit by lightning ($n \lambda(n) \gg 1$). 
On the other hand, our assumption $\lambda(n) \ll 1$ implies that small 
clusters are not likely to be hit by lightning, so they will be growing 
more-or-less as in the Erd{\H{o}}s-R{\'e}nyi model. The heuristic suggests 
that after the critical time the system remains critical forever. 
R{\'a}th and T{\'o}th \cite{RT09} proved that this is indeed the case.

\begin{theorem} \cite{RT09}
Suppose that $1/n \ll \lambda(n) \ll 1$. Let $\bar{v}^n_k(t)$ be the 
proportion of vertices in clusters of size $k$ at time $t$ in the
forest-fire evolution.\\
(i) There is a deterministic limit in probability
$\bar{v}_k(t) = \lim_{n \to \infty} \bar{v}^n_k(t)$.\\
(ii) If $t \le t_c = 1$, $\bar{v}_k(t) = v_k(t)$.\\
(iii) If $t \ge t_c = 1$ we have $\sum_{\ell \ge k} \bar{v}_\ell(t) \asymp k^{-1/2}$. 
\end{theorem}

\begin{remark}
Analogous statements hold for general initial conditions satisfying
a moment condition; see \cite{RT09}.
\end{remark}

The phenomenon that a state characterized by power laws is reached and 
then maintained by the dynamics is called \textbf{self-organized criticality}. 
The term was introduced by Bak, Tang and Wiesenfeld \cite{BTW88}, who suggested 
that this mechanism would be present in many physical systems for which 
power laws had been observed empirically. Examples include: energy release in 
earthquakes, avalanche sizes in rice- and sandpiles, areas of forest-fires
and many others; see the book \cite{Jensen}. Bak, Tang and Wiesenfeld used the sandpile
dynamics to illustrate their idea via numerical simulations \cite{BTW88}. 
After Dhar \cite{Dhar90} discovered the Abelian property and established the 
fundamental results discussed in Section \ref{sec:model}, the Abelian sandpile 
became the primary theoretical example of SOC \cite{Dhar06}.

\subsection{Self-organized criticality in the Abelian sandpile model}
\label{ssec:SOC-sandpile}

The following heuristic suggests that on a large graph, avalanches in 
the stationary sandpile Markov chain will occur on all scales up to the
size of the system. Start from an empty pile. Initially, when not many 
particles have been added yet, avalanches will be small. As more particles
get added, the typical size of avalanches grows. The only limit to this growth 
is that particles are lost to the sink. Hence when stationarity is reached, 
we can expect to see avalanches on all scales up to the size of the system.

Numerical simulations \cites{Dhar06,Manna90} of the model on subsets of 
$\Z^d$ suggest that the above heuristic is correct, and various
avalanche characteristics have power law distributions, up to a cut-off
that grows with the system size. In this section we state some conjectures 
that quantify this in terms of critical exponents. In our discussions, 
we consider the model on the wired graph $G_n$ constructed from a finite box 
$V_n = [-n,n]^d \cap \Z^d$, as in Example \ref{example:wired-graph}.

First we briefly comment on the case $d = 1$. Here one can explicitly 
compute the set of recurrent configurations and the sandpile group;
see Exercise \ref{ex:1D}. The sandpile group of $G_n$ is isomorphic to 
$\Z_{2n+2}$; in particular, the number of recurrent configurations grows only 
linearly in $n$. This is in contrast with $d \ge 2$, where the number of
recurrent configurations grows exponentially in $n^d$.
There is significantly ``less randomness'' in $d = 1$ than 
in $d \ge 2$. Avalanches can be computed explicitly in $d = 1$, 
and it is found that with probability approaching $1$ all avalanches 
reach the sink. On the other hand, for $d \ge 2$ most avalanches 
do not reach the sink. Below we restrict our attention to 
$d \ge 2$, and refer to \cites{Dhar06,MRSV00} for more details on 
the one-dimensional case.

We write $\nu_n$ for the stationary distribution of the sandpile 
Markov chain on $G_n$. Given $z \in V_n$, we define 
the \textbf{avalanche cluster at $z$} as the set of 
vertices that are toppled when we add a particle at $z$
to the sandpile $\xi$:
\eqn{e:Av}
{ \Av_{z,V_n}
  = \Av_{z,V_n}(\xi)
  := \{ x \in V_n : n(z,x,\xi) > 0 \}. }
We also define the \textbf{size} of the avalanche as
the number of topplings, with multiplicity:
\eqn{e:S}
{ S_{z,V_n}
  = S_{z,V_n}(\xi)
  := \sum_{x \in V_n} n(z,x;\xi). }
The \textbf{radius} of the avalanche is:
\eqn{e:R}
{ \mathrm{rad}(\Av_{z,V_n}(\xi))
  := \max \{ |x - z| : x \in \Av_{z,V_n}(\xi) \}. }

\subsubsection{Two easy critical exponents}
\label{sssec:easy}

We start with computing two easy exponents.
Recall Dhar's formula from Section \ref{ssec:stationary}.
Observe that 
\eqnst
{ \frac{1}{2d} (\Delta'_{V_n})_{zx} 
  = I_{zx} - \p^1_n(z,x), } 
where $\p^1_n(z,x)$ is the transition matrix of simple random 
walk on $V_n$ stopped on the first exit from $V_n$, and
$I$ is the identity matrix. Denote
\eqnsplst
{ \p^k_n(z,x)
  &:= \text{$k$-step transition probability from $z$ to $x$}; \\
  G_n(z,x)
  &:= \Big( \Delta'_{V_n} \Big)^{-1}_{zx}, \quad z, x \in V_n. }
The matrix $G_n$ has a well-known interpretation in terms of the 
simple random walk.

\begin{exercise}
\label{ex:Green}
(i)  Show that 
\eqnst
{ 2d\, G_n(z,x)
  = \sum_{k=0}^\infty \p^k_n(z,x)
  = \E^z [ \text{number of visits to $x$ before exiting $V_n$} ]. }
(ii) Show that if $z \in [-n/2,n/2]^d \cap \Z^d$, 
we have
\eqnst
{ \sum_{x \in V_n} G_n(z,x)
  \asymp n^2, \quad n \ge 1. }
See \cite{LLbook}*{Lemma 4.6.1 and Proposition 6.2.6}
\end{exercise}

Dhar's formula in the present setting says that 
\eqn{e:Dhar-sq}
{ \E_{\nu_n} [ n(z,x;\cdot) ]
  = G_n(z,x). }
Exercise \ref{ex:Green}(ii) gives that 
\eqn{e:n^2}
{ \E_{\nu_n} [ S_{o,V_n} ] 
  \asymp n^2, }
where we write $o$ to denote the origin in $\Z^d$.
In particular, in the stationary sandpile, the expected size of an 
avalanche started at $o$ diverges as $n \to \infty$. Compare this with 
the divergence of the expected cluster size for critical percolation.

Let $d \ge 3$. Then
\eqn{e:lim-Green}
{ \lim_{n \to \infty} G_n(z,x) 
  = G(z,x)
  = (2d)^{-1} \E^z [ \text{number of visits to $x$} ]
  < \infty } 
exists. The function $2d \, G(z,x)$ is called the \textbf{Green function} 
of the random walk. It is known that for all $d \ge 3$ the Green function 
is asymptotic to $a_d |x-z|^{2-d}$ as $|x-z| \to \infty$; see \cite{LLbook}*{Theorem 4.3.1}. 
Hence, by \eqref{e:Dhar-sq} and \eqref{e:lim-Green}, we have
\eqn{e:mean-toppling}
{ \lim_{n \to \infty} \E_{\nu_n} [ n(o,x; \cdot) ] 
  = G(o,x)
  \sim \frac{(2d)^{-1} a_d}{|x|^{d-2}} \quad \text{ as $|x| \to \infty$.} }

In order to neatly formulate asymptotic results as $n \to \infty$,
the following theorem will be useful. 

\begin{theorem}\cite{AJ04}*{Theorem 1}
Let $d \ge 2$. There is a measure $\nu$ on the space 
$\{ 0, 1, \dots, 2d - 1 \}^{\Z^d}$ such that $\nu_n \Rightarrow \nu$
in the sense of weak convergence.
\end{theorem}

Now \eqref{e:mean-toppling} can be rephrased in the simpler form
(see \cite{JR08}):
for all $d \ge 3$ we have 
\eqn{e:mean-toppling-nu}
{ \E_{\nu} [ n(o,x; \cdot) ]
  \sim \frac{(2d)^{-1} a_d}{|x|^{d-2}} \quad \text{ as $|x| \to \infty$.} }
Here the precise meaning of the random variable $n(o,x; \cdot)$ is as 
follows. Draw a sample configuration from the limiting measure $\nu$. 
Add a particle at $o$, and attempt to stabilize by toppling all unstable
sites simultaneously, whenever there are such. Then $n(o,x; \cdot)$ is 
the number of induced topplings at $x$. We write 
$\Av_z = \{ x \in \Z^d : n(z,x; \cdot) > 0 \}$ and
$S_z = \sum_{x \in \Z^d} n(z,x; \cdot)$.

\subsubsection{Further critical exponents}
\label{sssec:further}

The relation \eqref{e:mean-toppling-nu} gives the average 
number of topplings at $x$ induced by adding a particle at $o$.
We now state a theorem and a conjecture concerning the \emph{probability}
that $x$ topples, if we add a particle at $o$.

\begin{theorem}\cite{JRS13}
\label{thm:prob-toppling-nu}
For all $d \ge 5$ there are constants $c = c(d), C = C(d) > 0$ such that
\eqn{e:prob-toppling-nu}
{ \frac{c}{|x|^{d-2}} 
  \le \nu [ x \in \Av_o ]
  \le \frac{C}{|x|^{d-2}}, \quad x \not= 0. }
\end{theorem}

Note that the upper bound follows easily from Dhar's formula \eqref{e:mean-toppling-nu}
and Markov's inequality, so the real content of the theorem is the lower bound.

\begin{conjecture}
For $2 \le d \le 4$ there exists $\eta = \eta(d) \ge 0$ such that 
\eqnst
{ \nu [ x \in \Av_o ]
  = \frac{1}{|x|^{d-2 + \eta + o(1)}} \quad \text{ as $|x| \to \infty$.} }
\end{conjecture}

We have intentionally written the exponent in the form $d - 2 + \eta$,
in order to highlight the comparison with \eqref{e:prob-toppling-nu} 
and \eqref{e:exponents}. Upper bounds for the exponent $\eta$
in dimensions $d = 2, 3, 4$ were recently proved in \cite{BHJ17}.

The second conjecture we state concerns the number of topplings
in an avalanche. This can be measured by the total number 
of topplings, with or without multiplicity. 

\begin{conjecture}
For all $d \ge 2$ there exist exponents 
$\tau = \tau(d),\, \tau' = \tau'(d) \ge 0$ depending on $d$ such that 
\eqnsplst
{ \nu [ S_o \ge k ]
  &= k^{1-\tau + o(1)}, \quad \text{ as $k \to \infty$;} \\
  \nu [ |\Av_o| \ge k ]
  &= k^{1-\tau' + o(1)}, \quad \text{ as $k \to \infty$;} }
\end{conjecture}

Since $\E_\nu S_o = \infty$, we must have $\tau \le 2$, if it
exists. It is also plausible that $\E_\nu |\Av_o| = \infty$, 
so we should also have $\tau' \le 2$, if it exists. 
Manna \cite{Manna90} presents numerical evidence suggesting 
that $\tau = \tau'$. Theorem \ref{thm:prob-toppling-nu} gives
rigorous support to this conjecture when $d \ge 5$. Indeed,
\eqref{e:prob-toppling-nu} shows that  
\eqnst
{ \E_\nu [ n(o,x;\cdot) \,|\, n(o,x;\cdot) \ge 1 ]
  = O(1), }
that is, the average number of topplings at $x$, conditional on
the event that $x$ topples, is $O(1)$.
Heuristic arguments suggesting that 
$\tau = \tau' = 3/2$ when $d > 4$ were given by Priezzhev \cite{Pr00}. 
This has recently been proved Hutchcroft \cite{H18} in great 
generality that goes well beyond $\Z^d$.

Finally, we state a conjecture regarding the radius of an avalanche.

\begin{conjecture}
For all $d \ge 2$ there exists $\alpha \ge 0$ such that 
\eqnst
{ \nu [ \mathrm{rad}(\Av_o) > r ]
  = r^{-\alpha+o(1)}, \quad \text{ as $r \to \infty$.} }
\end{conjecture}

When $d > 4$, it was conjectured in \cite{Pr00} that $\alpha = 2$.
This has recently been proved by Hutchcroft \cite{H18}.
Upper bounds on the exponent $\alpha$ in $d = 2, 3, 4$, and lower bounds
in $d = 3, 4$ were given in \cite{BHJ17}.

It is a well-known heuristic in statistical physics that the behaviour 
of a lattice model in sufficiently high dimensions can be approximated
by its behaviour on an infinite $k$-regular tree with $k \ge 3$,
called the \textbf{Bethe lattice}. 
In general, there exists a \textbf{critical dimension} $d_c$ such that for $d > d_c$ 
the values of the critical exponents do not depend on $d$ and take on
the same values as on the $k$-regular tree. The conjectured critical
dimension for the sandpile model is $d_c = 4$ (for the percolation 
model of Section \ref{ssec:percolation} it is conjectured to be $d_c = 6$).
Rigorous support to the idea that $d_c = 4$ for the Abelian
sandpile model is provided by Theorem \ref{thm:Pemantle}, showing
that there is a clear change in behaviour at dimension $4$ for the
closely related wired spanning forest measure. 

When $\Z^d$ is replaced by a $k$-regular tree, the distribution of the 
random set $\Av_o$ has been computed explicitly by 
Dhar and Majumdar \cite{DM90} using combinatorial methods. 
In particular, they obtain $\tau = 3/2$. 
This has been recently extended to supercritical Galton-Watson trees
\cite{JRS18}. It is natural to define the 
Euclidean distance between vertices $z, x$ of the tree as the square 
root of their graph distance. With respect to this distance 
Dhar and Majumdar \cite{DM90} also obtain $\alpha = 2$.

\section{Connections to other models}
\label{sec:connections}

One of the appeals of the Abelian sandpile is its close relationship with 
other probability models on graphs. In this section we present 
connections to: spanning trees; the rotor-router walk;
the random cluster model and the Tutte polynomial. 

\subsection{The burning bijection}
\label{ssec:burning}

As in Section \ref{sec:model}, let $G = (V \cup \{ s \}, E)$ be 
a finite connected multigraph. The \textbf{burning algorithm} introduced
by Dhar \cite{Dhar90} provides an efficient way to check whether a
given stable sandpile is recurrent. This gives a combinatorial 
characterization of recurrent sandpiles. The algorithm leads
to the \textbf{burning bijection} between recurrent sandpiles on $V$ 
and spanning trees of $G$. This bijection is due to Majumdar and Dhar \cite{MD92}.

\begin{definition}
Let $\eta \in \Omega_G$, and let $\es \not= F \subset V$. We say that 
$\eta$ is \textbf{ample} for $F$, if there exists $x \in F$ such that 
$\eta(x) \ge \deg_F(x)$ (the degree of $x$ in the subgraph induced 
by the set of vertices $F$).
\end{definition}

\begin{lemma}
\label{lem:ample}
If $\eta \in \cR_G$, then $\eta$ is ample for all $\es \not= F \subset V$.
\end{lemma}

\begin{proof}
Due to Theorem \ref{thm:recurrent}(ii)(c), there exists $\xi$ such that 
$\xi^\circ = \eta$ and $\xi(x) \ge \deg_G(x)$ for all $x \in F$. 
Fix a stabilizing sequence for $\xi$. Observe that each vertex in $F$ 
must topple in this stabilizing sequence. Let $x$ be the first vertex 
among the vertices in $F$ that finishes toppling.
After $x$ finishes toppling, it receives particles from each of its 
neighbours in $F$ (as each of these neighbours will still topple). 
The number of particles received is altogether 
$\sum_{y \in F,\, y \not= x} a_{yx} = \deg_F(x)$.
Hence $\eta(x) \ge \deg_F(x)$, as required.
\end{proof}

\medbreak

\textbf{The burning algorithm.} The input of the algorithm is a 
stable sandpile $\eta \in \Omega_G$. 

At time $t = 0$, we declare $s$ ``burnt''. We set $B_0 = \{ s \}$,
the set of vertices burnt at time $0$, and set $U_0 = V$, the set of
vertices unburnt at time $0$.

At time $t = 1$, we declare burnt all $x \in U_0$ such that 
$\eta(x) \ge \deg_{U_0}(x)$. That is, we set
\eqnsplst
{ B_1
  &= \{ x \in U_0 : \eta(x) \ge \deg_{U_0}(x) \} \\
  U_1
  &= U_0 \setminus B_1. }

At a generic time $t \ge 1$, we declare burnt all vertices 
in $x \in U_{t-1}$ such that $\eta(x) \ge \deg_{U_{t-1}}(x)$. 
That is, we set 
\eqnsplst
{ B_t
  &= \{ x \in U_{t-1} : \eta(x) \ge \deg_{U_{t-1}}(x) \} \\
  U_t
  &= U_{t-1} \setminus B_t. }

\medbreak

We must eventually have $U_T = U_{T+1} = \dots$ for some $1 \le T < \infty$.
If $U_T \not= \es$, then the relation $U_T = U_{T+1}$ (equivalently:
$B_T = \es$) shows that $\eta$ is not ample for $U_T$. 
Hence Lemma \ref{lem:ample} shows that $\eta$ is not recurrent
in this case. We will see shortly the converse, that is, when $U_T = \es$
we can conclude that $\eta$ is recurrent.

\medbreak 

\textbf{The burning bijection.} Denote $\cT_G = \{ \text{spanning trees of $G$} \}$.
We use the burning algorithm to define a map $\varphi : \cR_G \to \cT_G$.
For every $x \in V$ fix an ordering $\prec_x$ of the edges incident
with $x$. It will be useful to think of these edges as being oriented from
$x$ towards the corresponding neighbours. Let $\eta \in \cR_G$. 
The spanning tree $\varphi(\eta)$ will be defined by assigning to each
$x \in V$ an edge incident with $x$, and oriented outwards from $x$.
The construction will guarantee that the edges form a spanning tree
oriented towards $s$.  
We just saw that running the burning algorithm on $\eta$ burns 
all vertices. Therefore, given any vertex $x \in V$, there is a 
unique $t = t(x) \ge 1$ such that $x \in B_t$. Let 
\eqnsplst
{ F_x 
  &= \{ f : \tail(f) = x,\, \head(f) \in B_{t-1} \} \\
  m_x
  &= \left| \Big\{ f : \tail(f) = x,\, 
     \head(f) \in \bigcup_{r \le t-1} B_r \Big\} \right|. }
Here $| \cdot |$ denotes the number of elements of a set.
Observe the following properties:\\
(i) We have $F_x \not= \es$. This is because a vertex becomes burnable 
in the algorithm precisely because a sufficient number of its neighbours
have burnt to satisfy the inequality $\eta(x) \ge \deg_{U_{t-1}}(x)$. Hence 
there always is at least one neighbour of $x$ that burned in the previous 
step.\\
(ii) We have $\deg_G(x) - m_x \le \eta(x) < \deg_G(x) - m_x + | F_x |$.
The first inequality holds because the left hand side is $\deg_{U_{t-1}}(x)$.
The second inequality holds because if it was violated, then $x$ should
have burnt at a time $\le t-1$.\\
Supposing $\eta(x) = \deg_G(x) - m_x + i$, with $0 \le i < | F_x |$, and
$F_x = \{ f_0 \prec_x f_1 \prec_x \prec_x \dots \prec_x f_{| F_x | -1} \}$,
we set $e_x = f_i$. Now put $\varphi(\eta) := \tau := \{ e_x : x \in V \}$.
Since $\head(e_x) \in B_{t(x)-1}$ for each $x$, the collection $\tau$ 
does not contain cycles. Therefore, it is a spanning tree of $G$
(oriented towards $s$). Now forget the orientation of the edges
to obtain an unoriented spanning tree. (Note: there is no loss of
information in doing so, since the orientation is uniquely recovered
by following paths leading to $s$).

\medbreak

\begin{exercise}
\label{ex:injective}
Show that $\varphi$ is injective. \emph{Hint:} If $\eta_1 \not= \eta_2$,
there is a first time $t$ when ``different things happen'' in the 
constructions of $\varphi(\eta_1)$ and $\varphi(\eta_2)$. Check that at this
time some $e_x$ is assigned differently for the two configurations.
See \cite{MD92}.
\end{exercise}

A well-known result in combinatorics is the Matrix-Tree Theorem \cite{Bbook}, 
that states that $|\cT_G| = \det ( \Delta'_G )$. We saw in Theorem \ref{thm:group}(i)
that also $|\cR_G| = \det ( \Delta'_G )$. Therefore Exercise \ref{ex:injective}
implies the following corollary.

\begin{corollary}
The mapping $\varphi : \cR_G \to \cT_G$ is a bijection.
\end{corollary}

\begin{exercise}
\label{ex:comb}
Deduce the following combinatorial characterization of recurrent
states:
\eqnst
{ \parbox{12cm}{a sandpile $\eta$ is recurrent if and only if 
    it is ample for any $\es \not= F \subset V$.} }
\emph{Hint:} The spanning tree $\varphi(\eta)$ is well-defined, and
injectivity still holds, whenever $\eta$ ``passes'' the burning 
algorithm, that is, when all vertices burn. 
See \cite{HLMPPW}*{Lemma 4.2}.
\end{exercise}

\begin{exercise}
\label{ex:1D}
Show that if $G_n = (V_n \cup \{ s \}, E_n)$ is the wired graph
constructed from $V_n = \{ 1, \dots, n \} \subset \Z$, then 
$\cR_{G_n}$ consists of those sandpiles for which there is 
at most one vertex with no particles (in particular, $|\cR_{G_n}| = n+1$).
Show that the sandpile group of $G_n$ is isomorphic to $\Z_{n+1}$,
and it is generated by $E_1$. See \cites{MRSV00,DRSV95}.
\end{exercise}

\begin{exercise}
\label{ex:inverse}
Specify explicitly the inverse map $\varphi^{-1} : \tau \mapsto \eta$.
\emph{Hint:} $x \in B_t$ if and only if $\dist_\tau(s,x) = t$,
where $\dist_\tau$ is the graph-distance in the tree $\tau$. 
See \cite{AJ04}.
\end{exercise}

Recall that the stationary measure $\nu_G$ is
the uniform distribution on recurrent sandpiles. The bijection 
$\varphi$ maps this measure to the uniform distribution 
on the set of spanning trees of $G$. This is called the
\textbf{uniform spanning tree} measure, denoted $\mathsf{UST}_G$.
See \cite{LPbook} and \cite{BLPS} for the rich theory of 
uniform spanning trees.

What makes the burning bijection a very useful tool, is that 
there is a simple (and indeed very efficient) algorithm due 
to Wilson \cite{W96} to generate a uniformly random element 
of $\cT_G$, that is, a sample from $\UST_G$. Mapping this random 
tree back via the map $\varphi^{-1} : \cT_G \to \cR_G$,
one can analyze the measure $\nu_G$. In order to describe
Wilson's algorithm, we need the procedure of loop-erasure.
Given a path $\pi = [w_0, w_1, \dots, w_k]$ in $G$, we define 
the \textbf{loop-erasure} $\LE(\pi) = [v_0, \dots, v_\ell]$
of $\pi$ by chronologically erasing loops from $\pi$, as
they are created. That is, we follow the steps of $\pi$ until
the first time $t$, if any, when $w_t \in \{ w_0, \dots, w_{t-1} \}$.
Suppose $w_t = w_i$. We remove the loop 
$[w_i, w_{i+1}, \dots, w_t = w_i]$ from $\pi$, and continue
tracing $\pi$. The process stops when there are no more loops
to remove, yielding a self-avoiding path denoted $\LE(\pi)$.
If $\pi$ is obtained from a random walk process on $G$, its 
loop-erasure is called the loop-erased random walk (LERW)
\cite{LLbook}.

\medbreak

\textbf{Wilson's algorithm.} Fix a vertex $r$ of $G$ (for example, 
in the sandpile context $r = s$ turns out to be a natural choice), 
and let $v_1, v_2, \dots, v_K$ be an arbitrary enumeration of the 
remaining vertices of $G$. Let $\tau_0 = \{ r \}$. Start a simple 
random walk on $G$ at the vertex $v_1$, and stop it when $r$ is first 
hit. We attach to $\tau_0$ the loop-erasure of the path from $v_1$ to 
$r$, and call the resulting path $\tau_1$. Now we start a second simple 
random walk from $v_2$, stop it when it hits $\tau_1$, and attach the 
loop-erasure to $\tau_1$. This gives a tree $\tau_2$. When we have 
visited all the vertices, we have a spanning tree $\tau_K$ of $G$. 
Wilson's theorem \cite{W96} shows that this tree is uniformly distributed
over all spanning trees of $G$.

\medbreak

The LERW and Wilson's algorithm are also very useful when we pass to 
infinite graphs (see \cite{BLPS}). In $\Z^d$, $d \ge 3$, the 
loop-erasure of an infinite simple random walk path is still
well-defined, because the path visits any vertex only finitely
often, due to transience. In $\Z^2$
the definition of the infinite LERW is not as straightforward.
One possible definition is take a LERW from the origin to the 
boundary of a ball of radius $n$, and take the weak limit of these
paths as $n \to \infty$ \cite{LLbook}.

\begin{exercise}
Give a direct proof (without appealing to the Matrix-Tree Theorem) 
that $\varphi : \cR_G \to \cT_G$ is a bijection. 
\emph{Hint:} See Exercises \ref{ex:injective} and \ref{ex:inverse}, 
the hint for Exercise \ref{ex:comb} and \cite{HLMPPW}*{Lemma 4.2}.
\end{exercise}

\subsection{The rotor-router model}
\label{ssec:rotor-router}

The rotor-router model, invented by Jim Propp, is a 
deterministic analogue of random walk \cite{HLMPPW}. It has also 
been discovered independently in the physics literature, where it 
is called the Eulerian walkers model \cite{PDDK96}. In the sandpile 
model, each vertex $x$ has to ``wait'' until it has collected 
$\deg(x)$ chips before it can send them to its neighbours.
The rotor-router mechanism allows us to send chips one-by-one.
The principal reference for this section is \cite{HLMPPW}.

The natural setting for the rotor-router walk is directed graphs. 
However, since the emphasis in this section is 
the connection to Abelian sandpiles, we will restrict to connected 
graphs of the form $G = (V \cup \{ s \}, E)$ as in Section \ref{sec:model},
and regard each edge of $E$ being present with both orientations.
For each $x \in V$, fix a cyclic
ordering of the edges incident with $x$, and orient these edges
outward from $x$. If $e$ is one of these edges ($\tail(e) = x$),
then we denote by $e^+$ the next edge in the cyclic ordering.

\begin{definition}
A \textbf{rotor configuration} is a choice of edges 
\eqnst
{ \rho
  = ( \rho(x) : x \in V ), }
such that $\rho(x) \in E$ and $\tail(\rho(x)) = x$ for each $x \in V$.
We think of $\rho(x)$ as the state of a rotor placed at the
vertex $x$. A \textbf{single-chip-and-rotor state} is a pair
$(\rho, w)$, where $w \in V$. We think of $w$ as the location 
of a chip placed on the graph. The \textbf{rotor-router operation}
advances the rotor at the current position $w$ according to the
cyclic ordering, and then moves the chip following the
new direction of the rotor. That is, we assign to $(\rho, w)$ 
the new state $(\rho^+, w^+)$, where
\begin{eqnarray*}
  \rho^+(y)
  &=& \begin{cases}
    (\rho(w))^+ & \text{if $y = w$;} \\
    \rho(y)     & \text{if $y \not= w$;} 
    \end{cases} \\
  w^+
  &=& \head(\rho^+(x)). 
\end{eqnarray*}  
Iterating the rotor-router operation gives the \textbf{rotor-router walk}.
When the chip arrives at the sink, we stop the walk, and remove the chip.
\end{definition}

\begin{lemma}
Starting from any single-chip-and-rotor state $(\rho,w)$, the rotor-router
walk eventually arrives at the sink. 
\end{lemma}

\begin{proof}
If $x \sim s$, then after at most $\deg_G(x)$ visits to $s$, the walk will
arrive at $s$. Inducting along a path to $s$, we obtain that any vertex
can only be visited finitely many times before the walk arrives at the sink.
\end{proof}

\begin{definition}
A \textbf{chip-and-rotor state} is a pair $t = (\rho, \eta)$, where 
$\rho$ is a rotor configuration and $\eta$ is a chip configuration 
(sandpile) on $V$.
If $\eta(x) \ge 1$, we say that \textbf{$x$ is active}. In this case,
by \textbf{firing $x$} we mean letting a single chip at $x$ take one
rotor-router step. $t$ is stable, if there is no active vertex (all have
moved to the sink).
\end{definition}

The next lemma shows that this model has an Abelian property.

\begin{lemma}
\label{lem:rotor-stabilize}
Starting from any chip-and-rotor state $(\rho, \eta)$, we reach the same
stable state eventually (that is when all chips arrived at the sink), 
regardless of what rotor-router steps we choose. 
\end{lemma}

\begin{proof}
This can be proved using similar ideas as Theorem \ref{thm:stabilization}.
\end{proof}

\begin{definition}
We denote by $\eta(\rho)$ the result of adding chips to the rotor configuration
$\rho$ according to $\eta$ and stabilizing. The \textbf{chip addition operator}
$E_x$ is defined on a rotor configuration $\rho$ as the result of
adding a single chip at $x$ and stabilizing, that is: $\mathbf{1}_x(\rho)$.
\end{definition}

\begin{definition}
A rotor configuration $\rho$ is \textbf{acyclic}, if the edges $(\rho(x) : x \in V)$
form a spanning tree of $G$ (oriented towards $s$, necessarily).
\end{definition}

\begin{lemma}
\label{lem:permute}
For any chip configuration $\eta$, the map $\rho \mapsto \eta(\rho)$
permutes the collection of acyclic rotor configurations.
\end{lemma}

Perhaps the cleanest way to prove this is to consider unicycles on 
strongly connected graphs; see \cite{HLMPPW}. For our specific
setting, the next two exercises sketch a proof.

\begin{exercise}
Show that in $\eta(\rho)$, each rotor is either in its original
position $\rho(x)$, or it points in the direction of the last chip
emitted from $x$. Conclude: if $\rho$ is acyclic, so is $\eta(\rho)$.
See \cite{HLMPPW}*{Section 3}
\end{exercise}

\begin{exercise}
Show that $\rho' \mapsto E_x(\rho')$, as a map from the collection of
acyclic rotor configurations to itself, is surjective for any $x \in V$.
The following steps can be used: Given $\rho$, add an oriented edge 
$(s,x)$ to $\rho$.\\
(i) There is an oriented cycle starting at $s$, let $(y_1,s)$ be its 
last edge. Place a chip at $y_1$, and move back the rotor at $y_1$ 
by one step. \\
(ii) Now there is an oriented cycle starting at $y_1$, let 
$(y_2,y_1)$ be its last edge. Move the chip back to $y_2$, and 
move back the rotor at $y_2$ by one step. \\
(iii) Show that eventually the chip arrives at $x$ and if the rotor
configuration at that time is $\rho'$, then we have $E_x \rho' = \rho$.\\
See \cite{HLMPPW}*{Section 3}.
\end{exercise}

\begin{theorem} \ \\
(i) The map $(\rho, [\eta]) \mapsto \eta(\rho)$ defines an action of the
sandpile group on acyclic rotor configurations.\\
(ii) The action is transitive, that is, for any acyclic $\rho, \rho'$
there exists $\eta$ such that $\eta(\rho) = \rho'$.\\
(iii) The action is free, that is, if $\eta(\rho) = \rho$ for some 
acyclic $\rho$ then $[\eta] = [0]$.
\end{theorem}

\begin{proof}
(i) From Lemma \ref{lem:rotor-stabilize} it is clear that 
$\eta_2( \eta_1 ( \rho ) ) = (\eta_1 + \eta_2) (\rho)$. Suppose
$\eta_1 \sim \eta_2$. We show that $\eta_1 (\rho) = \eta_2 (\rho)$. 
If $\eta(x) \ge \deg_G(x)$, we can advance $\deg_G(x)$ chips at $x$,
one along each edge incident with $x$, and leave the rotor at $x$ 
unchanged. It follows from this that $\eta(\rho) = \eta^\circ (\rho)$ 
for any chip configuration $\eta$. Let $I \in \cR_G$ be the sandpile
corresponding to the identity (i.e.~$[I] = [0]$). Then we have
\eqnst
{ I ( I (\rho) ) 
  = (I + I) (\rho) 
  = (I + I)^\circ (\rho)
  = I (\rho) }
for all acyclic rotor configurations $\rho$. Due to 
Lemma \ref{lem:permute}, $\{ I(\rho) : \text{$\rho$ acyclic} \}
= \{ \rho : \text{$\rho$ acyclic} \}$, and it follows that 
$I(\rho) = \rho$ for all acyclic $\rho$. Now we have
$(\eta_1 + I)^\circ = (\eta_2 + I)^\circ$, and 
\eqnst
{ \eta_i (\rho) 
  = I ( \eta_i (\rho) ) 
  = (I + \eta_i) (\rho)
  = (I + \eta_i)^\circ, \quad i = 1, 2. }
This implies the claim.

(ii) Given $\rho$, $\rho'$, let $0 \le \alpha(x) < \deg_G(x)$ be
the number of turns the rotor at $x$ has to make from position
$\rho(x)$ to $\rho'(x)$. Adding chips to $\rho$ according to $\alpha$ 
and letting each chip take a single step we obtain a
chip-and-rotor state of the form: $(\rho', \beta)$. 
Choose a chip configuration $\sigma$ such that 
$[\sigma] = [-\beta]$ (the inverse of $\beta$ in the sandpile
group). Let $\eta = \alpha + \sigma$. Then we have
\eqnst
{ \eta(\rho) 
  = (\sigma + \alpha)(\rho) 
  = (\sigma + \beta) (\rho')
  = \rho'. }
This proves transitivity of the action.

(iii) Suppose $\eta$ is a chip configuration, $\rho$ is acyclic, and
$\eta(\rho) = \rho$. This means that adding chips according to $\eta$
the rotor at $x$ makes a non-negative integer $c_x$ number of full turns
during stabilization. Since all chips arrive at the sink, $\eta(x)$
equals the number of chips emitted from $x$ minus the number of chips 
received at $x$, for each $x \in V$. Therefore:
\eqnst
{ \eta(x) 
  = \deg_G(x) c_x - \sum_{y \in V} a_{yx} c_y
  = \sum_{y \in V} c_y \Delta'_{yx}, \quad x \in V. }
This shows that $[\eta] = [0]$. 
\end{proof}

\begin{remark}
The above proof does not rely on the Matrix-Tree Theorem, and 
in fact provides a new proof of it; see \cite{HLMPPW}*{Corollary 3.18}.
\end{remark}

\begin{remark}
Regarding acyclic rotor configurations as spanning trees of $G$,
the action of $K_G$ allows one to view the sandpile Markov chain
as a dynamics on trees. This dynamics on trees seems more transparent 
and explicit than the one obtained using the burning bijection.
\end{remark}

\begin{open}
Is there a meaningful link between avalanches in the Abelian sandpile
and either the rotor-router dynamics on spanning trees or the dynamics 
induced by the burning bijection?
\end{open}

\subsection{The random cluster model / Tutte polynomial}
\label{ssec:Tutte}

The uniform spanning tree measure $\mathsf{UST}_G$ is a limiting
case of the so-called random cluster measure. The random cluster 
model is a generalization of percolation. The relationship between
sandpiles and the random cluster measure leads to 
a formula for the generating function of recurrent sandpiles
enumerated by their total number of particles. In this section 
again $G = (V \cup \{ s \}, E)$ is a finite connected multigraph.

\begin{definition}
If $\eta \in \cR_G$, the \textbf{mass of $\eta$} is defined
as $m(\eta) = \sum_{x \in V} \eta(x)$. We put
$N_m = | \{ \eta \in \cR_G : m(\eta) = m \} |$, and
let $\cN(y) = \sum_{m} N_m y^m$ be the generating function 
according to mass.
\end{definition}

\begin{exercise}
Show that for all $\eta \in \cR_G$ we have:
\eqnst
{ |E| - \deg_G(s) 
  \le m(\eta)
  \le 2 |E| - \deg_G(s) - |V|. }
Show that the lower bound is achieved for any $\eta \in \cR_G$
that is \emph{minimal} in the sense that $\eta - \mathbf{1}_x \not\in \cR_G$
for all $x \in V$. \emph{Hint for the lower bound:} Use the burning algorithm.
See \cite{Dhar06}*{Section 7.2}
\end{exercise}

The \textbf{random cluster model} on $G$ has two parameters:
$0 < p < 1$ and $q > 0$. It is specified by a probability measure
$\Pr_{p,q}$ on the space $\{ E ' : E' \subset E \}$, that is 
given by:
\eqnst
{ \Pr_{p,q} [ E' ]
  = \frac{1}{Z_{p,q}} p^{|E'|} (1 - p)^{|E|-|E'|} q^{k(E')}, }
where $k(E')$ denotes the number of connected clusters in the 
edge-configuration $E'$, and $Z_{p,q}$ is a normalizing factor
to make $\Pr_{p,q}$ a probability measure. Observe that when
$q = 1$, we get the percolation model. 

\begin{exercise}[See \cite{Grbook2}*{Theorem 1.23}]
As $p \to 0$ and $q/p \to 0$ we have $\Pr_{p,q} \to \mathsf{UST}_G$.
\end{exercise}

Abbreviating $v = p/(1-p)$, we have
\eqnsplst
{ Z_{p,q}
  &= \sum_{E' \subset E} p^{|E'|} (1 - p)^{|E|-|E'|} q^{k(E')}
  = (1 - p)^{|E|} \sum_{E' \subset E} v^{|E'|} q^{k(E')} \\
  &=: (1 - p)^{|E|} Z'_{v,q}. }
Letting $q \downarrow 0$, the dominant terms are the ones
with $k(E') = 1$, that is the ones where $E'$ is connected.
The number of edges in such a graph is at least $|V|$
(with equality for spanning trees). Hence we can write:
\eqnst
{ Z'_{v,q}
  = q v^{|V|} H(v) + O(q^2), \quad \text{ as $q \downarrow 0$,} }
where $H(v)$ is a polynomial. Note that $H(0)$ equals the 
number of spanning trees of $G$. 
The following theorem is due to Merino L{\'o}pez \cite{ML97}.

\begin{theorem}
\label{thm:Merino}
We have
\eqnst
{ \cN(y)
  = y^{|E|-\deg_G(s)} H(y-1). }
\end{theorem}

See \cite{Dhar06} for a proof that follows the ideas of the 
burning algorithm to associate a recurrent configuration to 
groups of connected graphs $E'$.

One has $H(y-1) = T(1,y;G)$, where $T(x,y;G)$ is the so-called
Tutte polynomial of $G$, a well-known graph invariant in combinatorics
\cite{Bbook}. More generally, $Z_{p,q}$ can be expressed 
in terms of the Tutte polynomial; see \cite{Grbook2}*{Section 3.6}.

\begin{exercise}
According to Theorem \ref{thm:Merino}, the number of 
recurrent sandpiles that have a minimal number of particles
is $\cN(0) = H(-1) = T_G(1,0)$. It is known \cite{Bbook}
that $T_G(1,0)$ counts the number of \emph{acyclic orientations} 
of $G$ with a unique sink at a fixed vertex of $G$. Taking the 
unique sink to be at $s$, use the burning algorithm to 
construct an explicit bijection between 
\emph{minimal sandpiles} and \emph{acyclic orientations}
of $G$ with unique sink at $s$.
\end{exercise}

\section{Determinantal formulas and exact computations}
\label{sec:determinantal}

In this section we will see that certain sandpile probabilities 
can be expressed in terms of determinants, and in some
cases these can be evaluated explicitly. The fundamental fact behind 
this is that all finite-dimensional marginals of the uniform spanning tree 
admit a determinantal formula.

\subsection{The Transfer-Current Theorem}

Let $G$ be a finite connected (unoriented) graph. Write $T_G$ for a random 
spanning tree of $G$ chosen uniformly. The following 
theorem is due to Burton and Pemantle.

\begin{theorem}[\textbf{Transfer Current Theorem} \cite{BP93}]
\label{thm:TCT}
There exists a matrix $Y_G$ such that for any $k \ge 1$ and 
distinct edges $e_1, \dots, e_k$ of $G$ we have
\eqn{e:transfer-current}
{ \Pr [ e_1, \dots, e_k \in T_G ]
  = \det ( Y_G(e_i, e_j) )_{i,j=1}^k. }
\end{theorem}

The simplest definition of the \textbf{transfer-current matrix} $Y_G$, 
is in terms of random walk. (See \cite{LPbook} for a definition of $Y_G$
in terms of electrical networks.) Given \emph{oriented} edges $e,f$ of $G$,
consider the simple random walk on $G$ started at $\tail(e)$ and stopped
when it first hits $\head(e)$. Let $J^e(f)$ be the expected net usage of $f$
by the walk, i.e. the number of times $f$ was used
minus the number of times the reversal of $f$ was used. 
Then $Y_G(e,f) = J^e(f)$. Note that this requires us to chose an orientation
for each edge appearing in the right hand side of the Transfer Current Theorem,
whereas in the left hand side the edges are unoriented. It is part of the
statement of the theorem that the right hand side is independent of what
orientations are chosen.\footnote{This can also be checked directly 
using the time reversal of the simple random walk on $G$.}
Due to the structure present in \eqref{e:transfer-current}, the random 
collection of edges $T_G$ is called a \textbf{determinantal process} 
with \textbf{kernel} $Y_G$.
There is an extension of \eqref{e:transfer-current} to all
cylinder events, also due to \cite{BP93}. A simple case of it that we will
use later is:
\eqnst
{ \Pr [ e_1, \dots, e_k \not\in T_G ]
  = \det ( K_G ( e_i, e_j ) )_{i,j=1}^k, }
where $K_G = I_G - Y_G$, with $I_G$ the identitiy matrix.
See the survey \cite{HKPV06} for more information on determinantal processes.

\subsection{The height $0$ probability}
\label{ssec:height-0}

No simple formula like \eqref{e:transfer-current} is known for the 
finite-dimensional marginals of the sandpile measure $\nu_G$. However,
there is a method due to Majumdar and Dhar \cite{MD91} for the 
computation of the probabilities of \emph{minimal configurations}.
The simplest example is computing the probability that $\eta(o) = 0$, 
that we now explain.

Consider $V_n = [-n,n]^2 \cap \Z^2$, and let $G_n$ be the wired graph 
obtained from $V_n$. We write $\nu_n = \nu_{G_n}$ for short.
We will obtain a formula for $\nu_n [ \eta(o) = 0 ]$ that makes it
possible to compute its limit as $n \to \infty$.

Let $j_1, j_2, j_3$ denote the south, west, north neighbours of the origin, 
respectively. Let $G'_n$ be the graph obtained from $G_n$ 
by deleting the edges $\{ o, j_i \}$, $i = 1, 2, 3$. 
Given $\eta \in \cR_{G_n}$, let 
\eqnst
{ \eta'(y)
  := \begin{cases}
     \eta(y) - 1 & \text{if $y = j_1, j_2, j_3$;} \\
     \eta(y)     & \text{otherwise.}
     \end{cases} }

\begin{exercise}
\label{ex:reduce}
Show that 
\eqnst
{ \eta \in \cR_{G_n},\, \eta(o) = 0 \qquad\qquad \text{ if and only if } \qquad\qquad
  \eta' \in \cR_{G'_n}. }
\emph{Hint:} Use the burning algorithm. See \cite{MD91}.
\end{exercise}

We write $\Delta'_{G'_n}$ in the form $\Delta'_{G'_n} = \Delta'_{G_n} + B$. 
Note that the matrix $B$ has nonzero entries only in the rows and columns 
corresponding to $\{ o, j_1, j_2, j_3 \}$, and these are:
\eqn{e:Bmatrix}
{ \begin{blockarray}{ccccc}
  o & j_1 & j_2 & j_3 &  \\
  \begin{block}{(cccc)c}
 -3 &   1 &  1  &   1 & o \\
  1 &  -1 &  0  &   0 & j_1 \\
  1 &   0 & -1  &   0 & j_2 \\
  1 &   0 &  0  &  -1 & j_3 \\
  \end{block}
  \end{blockarray} }
The above allows us to write
\eqnspl{e:calculate}
{ \nu_n [ \eta(o) = 0 ]
  &= \frac{| \{ \eta \in \cR_{G_n} : \eta(o) = 0 \} |}{|\cR_{G_n}|} 
  \stackrel{\text{Ex.~\ref{ex:reduce}}}{=}
  \frac{|\cR_{G'_n}|}{|\cR_{G_n}|} 
  = \frac{\det ( \Delta'_{G_n} + B )}{\det ( \Delta'_{G_n} )} \\
  &= \det ( I + B (\Delta'_{G_n})^{-1} ). }
Due to the fact that $B$ is $0$ apart from the entries shown in \eqref{e:Bmatrix},
the determinant on the right hand side of \eqref{e:calculate}
reduces to a $4 \times 4$ determinant. Recall that 
$(\Delta'_{G_n})^{-1}_{xy} = G_n(z,w)$. Since 
the random walk is recurrent in two dimensions, 
$\lim_{n \to \infty} G_n(z,w) = \infty$.
Hence in order to take the limit $n \to \infty$, we need to 
rely on cancellations. 

In two dimensions the \textbf{recurrent random walk potential kernel} is defined as
\eqnst
{ a(x)
  = \lim_{N \to \infty} \sum_{k=0}^{N} \left[ \p^k(o,o) - \p^k(o,x) \right], }
where $\p^k(z,x)$ is the $k$-step transition probability of
simple random walk on $\Z^2$. See \cite{LLbook}*{Section 4.4.1}
for a proof that the limit exists and for further background.
Note that $a(o) = 0$. It holds that 
$\frac{1}{4}(\Delta a) (x) = - \mathbf{1}_o(x)$ (see \cite{LLbook}*{Proposition 4.4.2});
in particular, $a$ is a discrete harmonic function in
$\Z^2 \setminus \{ o \}$.

The potential kernel is related to $G_n$ by the following lemma
that is well-known. 

\begin{lemma}
\label{lem:pot-kern} 
For all $x \in \Z^2$, we have
\eqnst
{ A(x)
  := \frac{1}{4} a(x) 
  = \lim_{n \to \infty} [ G_n(o,o) - G_n(o,x) ]. }
\end{lemma}

Since we are going to prove a stronger version of this statement 
in Lemma \ref{lem:estimates}, Eqn.~\eqref{e:Gn(z,o)}, we omit the proof.

The values of $A(x)$ can be computed recursively from symmetry considerations
and the facts that:\\
(i) $\frac{1}{4} \Delta A(x) = - \frac{1}{4} \mathbf{1}_o(x)$;\\
(ii) $A((n,n)) = \frac{1}{\pi} \left[ 1 + \frac{1}{3} + \dots + \frac{1}{2n-1} \right]$;\\
see for example \cite{LPbook} or \cite{Spbook}*{Section 15}.
In particular, 
\begin{align}
\label{e:pot-kern}
  A(o) &= 0 &
  A(j_1) &= \frac{1}{4} &
  A(j_1+j_2) &= \frac{1}{\pi} \\
\notag
  A(j_1-j_3) &= 1 - \frac{2}{\pi} &
  & & 
  A(j_1 + j_2 - j_3) &= \frac{-1}{4} + \frac{2}{\pi}. 
\end{align}

Let us return to the limit of the determinant in \eqref{e:calculate}. 
Since the row sums of $B$ are $0$, the computation can be recast in terms of $A$. 
For example, the $o,o$ entry of $I + B (\Delta'_{G_n})^{-1}$ equals
\eqnst
{ 1 - 3 G_n(o,o) + G_n(o,j_1) + G_n(o,j_2) + G_n(o,j_3)
  \stackrel{n \to \infty}{\longrightarrow} 1 - 3 A(j_1)
  = \frac{1}{4}. }
Straightforward calculations using the values \eqref{e:pot-kern} and symmetry yield:
\eqn{e:p(0)}
{ p(0)
  := \lim_{n \to \infty} \nu_n [ \eta(o) = 0 ] 
  = \nu [ \eta(o) = 0 ] 
  = \frac{2}{\pi^2} - \frac{4}{\pi^3}. }

When $d \ge 3$, similar arguments apply. As
$\lim_{n \to \infty} G_n(z,x) = G(z,x)$, we have that 
$\nu [ \eta(o) = 0 ]$ is expressed
in terms of the Green function $G(z,x)$.

\subsection{The height $0$-$0$ correlation}
\label{ssec:corr-0-0}

The idea of Majumdar and Dhar presented in the previous section also gives 
a formula for the covariance between the events $\{ \eta(o) = 0 \}$ and 
$\{ \eta(y) = 0 \}$; see \cite{MD91}. Consider first 
two dimensions. This time we modify the
graph both near $o$ and $y$, by removing the edges leading
from $o$ to $j_1$, $j_2$, $j_3$, and the edges leading from
$y$ to neighbours $j'_1$, $j'_2$, $j'_3$.
Then similarly to the previous section,
$\nu_n [ \eta(o) = 0,\, \eta(y) = 0 ]$ can be
written as an $8 \times 8$ determinant, that arises from four
blocks of size $4 \times 4$. Since the row sums of $B$ are $0$,
row and column operations can be used to reduce the size of the blocks 
to $3 \times 3$. The result of this can be written in the 
form (see for example \cite{Durre}):
\eqn{e:det-3x3}
{ \nu_n [ \eta(o) = 0,\, \eta(y) = 0 ]
  = \det \left( I_{v = w} - K_n(v,w) \right)_{v,w \in \{o,y\}}, }
where $I_{v=w}$ is the $3 \times 3$ identity matrix when $v = w$ and the
$3 \times 3$ null matrix when $v \not= w$. The $3 \times 3$ matrix
$K_n(v,w)$ is given by:
\eqn{e:K_n}
{ K_n(v,w)
  = \left( \partial^{(1)}_{e} \partial^{(2)}_{f} G_n(v,w) \right)_{e, f}, }
where for any function $h : \Z^2 \to \R$ and vector $e \in \Z^2$
we define:
\eqnsplst
{ \partial^{(1)}_e h(v,w)
  &= h(v+e,w) - h(v,w), \\
  \partial^{(2)}_f h(v,w)
  &= h(v,w+f) - h(v,w). }
In the formula \eqref{e:K_n}, the vectors $e$ and $f$ range over the 
unit vectors: $(0,-1)$, $(-1,0)$ and $(0,1)$ (these are the vectors 
pointing from $o$ to $j_1$, $j_2$ and $j_3$).

Letting $n \to \infty$, we obtain an expression 
\eqn{e:det-3x3-asymp}
{ \nu [ \eta(o) = 0,\, \eta(y) = 0 ]
  = \det \left( I_{v = w} - K(v,w) \right)_{v,w \in \{o,y\}}, }
with
\eqn{e:K}
{ K(v,w)
  = \left( \partial^{(1)}_{e} \partial^{(2)}_{f} A(w-v) \right)_{e, f}. }
Here $\det(K(o,o)) = \det(K(y,y)) = p(0)$, from the previous
section. In order to understand how correlations decay as $|y| \to \infty$,
let us examine the order of magnitude of the entries of $K(o,y)$.
It is well know (see \cite{LLbook}*{Theorem 4.4.4}) that there 
exists a constant $c_0$ such that 
\eqn{e:pot-kern-asymp}
{ A(y) 
  = \frac{1}{2 \pi} \log|y| + c_0 + O\left( |y|^{-2} \right), \quad
    \text{as $|y| \to \infty$.} }
This shows that the entries of $K(o,y)$ (and that of $K(y,o)$) 
are $O(|y|^{-2})$, and hence 
\eqn{e:corr-asymp}
{ \nu [ \eta(o) = 0,\, \eta(y) = 0 ] - \nu [ \eta(o) = 0 ] \nu [ \eta(y) = 0 ]
  = O \left( |y|^{-4} \right), \quad \text{as $|y| \to \infty$.} }
One can compute the constant in this asymptotics using more precise
information on the error term in \eqref{e:pot-kern-asymp}. 
Indeed, regarding $y$ as a complex number, the error term 
in \eqref{e:pot-kern-asymp} is of the form
\eqnst
{ \frac{\mathfrak{Re}\, y^4}{|y|^6} + O \left( |y|^{-4} \right); }
see \cite{FU96} or \cite{KS04}.
Therefore, after taking second differences, the error term 
of \eqref{e:pot-kern-asymp} does not contribute to the $|y|^{-4}$ 
term in \eqref{e:corr-asymp}. 
This yields the result obtained by Majumdar and Dhar \cite{MD91}:
\eqn{e:cov-0-0}
{ \nu [ \eta(o) = 0,\, \eta(y) = 0 ] - \nu [ \eta(o) = 0 ] \nu [ \eta(y) = 0 ]
  \sim - \frac{p(0)^2}{2 |y|^4}, \quad \text{as $|y| \to \infty$.} }

In dimensions $d \ge 3$ a similar computation can be carried out
showing that the covariance between two $0$'s decays as
$-c |y|^{-2d}$, as $|y| \to \infty$, with $c = c(d) > 0$.

\subsection{Scaling limit of the height $0$ field}
\label{ssec:scaling-0}

The second differences of discrete Green functions considered in the 
previous section converge, under rescaling, to partial derivatives
of continuous Green functions. This allows to get formulas for the
scaling limit of the covariance functions between heights $0$.
The result is especially interesting in two dimensions, as there the
continuous Green function is conformally invariant, which implies
that the covariance functions transform in a nice way under
conformal maps. Although this fact seems to be well-known by physicists 
(see for example \cite{IP98}*{Section 3.3} and \cite{JPR06}), 
we are not aware of a mathematically precise formulation of it in the
physics literature. We state below a theorem of D\"{u}rre \cite{Durre} 
that provides such a formulation.

Let $U \subset \mathbb{C}$ be a bounded connected domain with smooth 
boundary. Let $U_\eps = (U/\eps) \cap \Z^2$, and for $v \in U$
let $v_\eps \in U_\eps$ be such that $|v/\eps - v_\eps| \le 2$.
Denote $h_\eps (v) = \mathbf{1}_{\eta(v_\eps) = 0}$, which is a random field,
indexed by $v \in U$, under the measure $\nu_{U_\eps}$.

\begin{theorem}[\cite{Durre}*{Theorem 1}]
Let $V \subset U$ be a finite set of points in the interior of $U$.  Then as
$\eps \to 0$, the rescaled joint moments
\eqnst
{ \eps^{-2 |V|} \E_{\nu_{U_\eps}} \left[ \prod_{v \in V} 
     \left[ h_\eps(v_\eps) - \E_{\nu_{U_\eps}} h_\eps(v) \right] \right] }
have a finite limit $E_U(v : v \in V)$, which is conformally covariant
with scale dimension $2$.
\end{theorem}

Here \textbf{conformally covariant} means that if $f: U \to U'$ is a conformal map,
then 
\eqnst
{ E_{U}(v : v \in V) 
  = E_{U'}(f(v) : v \in V) \cdot \prod_{v \in V} \left| f'(v) \right|^2, }
and the exponent $2$ is the \textbf{scale dimension}.
When $V = \{ v, w \}$, the limit is:
\eqnsplst
{ &E_U(v,w) \\
  &\quad = -c \left[ \left( \partial^{(1)}_x \partial^{(2)}_x g_U \right)^2
    + \left( \partial^{(1)}_y \partial^{(2)}_y g_U \right)^2
    + \left( \partial^{(1)}_x \partial^{(2)}_y g_U \right)^2
    + \left( \partial^{(1)}_y \partial^{(2)}_x g_U \right)^2 \right], }
where $g_U(v,w) = g_U((x_1,y_1),(x_2,y_2))$ is the continuous Green function in $U$ 
for the Laplacian with Dirichlet boundary conditions.

Summability of the covariance function $-c/|y|^4$ of \eqref{e:cov-0-0} 
suggests that if the random field $h_\eps$ is integrated against 
smooth test functions then we get a Gaussian limit. This is indeed the case.

\begin{theorem}[\cite{Durre}*{Theorem 3}]
There is a constant $\mathcal{V} > 0$ such that the following holds.
Let $f_i \in C_0^\infty(U)$, $1 \le i \le n$. Then the random variables
\eqnst
{ f_i \diamond h_\eps
  := \frac{\eps}{\sqrt{\mathcal{V}}} \sum_{v \in U_\eps}
     f_i(\eps v) \left( h_\eps(v) - \E_{\nu_{U_\eps}} h_\eps(v) \right) }
converge in distribution to a multivariate normal random variable with
covariance 
\eqnst
{ C_{ij} 
  = \int_U f_i(x,y) f_j(x,y) \, dx \, dy, \quad i, j = 1, \dots, n. }
\end{theorem}

\subsection{The probabilities of heights $1$, $2$, $3$ in two dimensions}
\label{ssec:height-123}

The probabilities of heights different from $0$ are, in general, more difficult to compute.
In the case of an infinite regular tree all height probabilities can be computed
using combinatorial methods; see \cite{DM90}. However, on Euclidean lattices
of dimension at least $2$, exact results are only known when $d = 2$. 
The goal of this section is to sketch the main ideas of Priezzhev \cite{Pr94} 
that yield the probabilities of heights $1$, $2$, $3$ on $\Z^2$. 
This section is quite long and technical in many parts, so the reader might 
want to skip some of the proofs on first reading.

\subsubsection{Background}

Let us denote
\eqnst
{ p(i) 
  := \nu [ \eta(o) = i ], \quad i = 0, 1, \ldots, 2d-1. }
In the case $d = 2$, Priezzhev \cite{Pr94} gave exact formulas for $p(1), p(2), p(3)$. 
He was able to express them in terms of explicit rational polynomials in $1/\pi$ 
and two multiple integrals. Grassberger evaluated the integrals numerically 
(see \cite{Dhar06}*{Section 9.3.1}), and observed that 
mysteriously the \emph{average height} $\zeta = \sum_{i=0}^3 i\, p(i)$ 
appears to be the simple rational number $17/8$. 

Jeng, Piroux and Ruelle \cite{JPR06} extended Priezzhev's ideas, and 
in particular, were able to express the $p(i)$'s in terms of a single integral. 
They noticed that numerical evaluation of the unknown integral gave $1/2 \pm 10^{-12}$, 
and conjectured that this integral is exactly equal to $1/2$. 
Assuming this conjecture, and combined with Priezzhev's work, 
they obtained the remarkable formulas:
\eqnspl{e:p(i)'s}
{ p(0)
  &= \frac{2}{\pi^2} - \frac{4}{\pi^3} \\
  p(1)
  &= \frac{1}{4} - \frac{1}{2 \pi} - \frac{3}{\pi^2} + \frac{12}{\pi^3} \\
  p(2)
  &= \frac{3}{8} + \frac{1}{\pi} - \frac{12}{\pi^3} \\
  p(3)
  &= 1 - p(0) - p(1) - p(2) 
  = \frac{3}{8} - \frac{1}{2 \pi} + \frac{1}{\pi^2} + \frac{4}{\pi^3}. }
These values indeed yield $\zeta = 17/8$ as the average height. 

Poghosyan and Priezzhev \cite{PP11} observed that the average height 
can be re\-phrased in terms of the LERW. Let 
\eqnst
{ \xi 
  = \E [ \text{number of neighbours of $o$ visited by the infinite LERW} ]. }
Then the statement $\zeta = 17/8$ is equivalent to $\xi = 5/4$ (this
equivalence will be explained below). Levine and Peres \cite{LP13} 
called $\xi$ the \textbf{looping constant of $\Z^2$}, and proved 
further relations between $\xi$, the number of spanning uni-cycles
and the Tutte-polynomial. The relations they prove hold 
in all dimensions $d \ge 2$.

Kenyon and Wilson \cite{KW11}, and independently, 
Poghosyan, Priezzhev and Ruelle \cite{PPR11} gave different proofs that 
$\xi = 5/4$, which in turn gives a rigorous confirmation that the 
aforementioned integral is exactly $1/2$ and proves the values \eqref{e:p(i)'s}.
A direct evaluation of the integral was given by 
Caracciolo and Sportiello \cite{CS12}. 
Subsequently, Kassel and Wilson \cite{KW14} gave a more direct proof of the
sandpile density. Kenyon and Wilson develop a general method for
calculating the probability that the infinite LERW in two dimensions
passes through any given vertex or any given oriented edge of $\Z^2$.
In principle, their method can be used to calculate all finite-dimensional
marginals of $\nu$.
The proof of Poghosyan, Priezzhev and Ruelle proceeds via a
connection to monomer-dimer coverings. They reduce the 
problem of $\xi = 5/4$ to calculating the probabilities of
certain local events in the monomer-dimer model that can be expressed in 
terms of finite determinants akin to the calculations in 
Section \ref{ssec:height-0}.

\subsubsection{The looping constant}
\label{sssec:looping}

Let us see the connection between the average height and the looping
constant. It will be convenient at this point
to introduce a slightly different 
version of the burning bijection. In what follows, let 
$G_n = (V_n \cup \{ s \}, E_n)$ be the
wired graph obtained from $V_n = [-n,n]^d \cap \Z^d$, $d \ge 2$.

\medbreak

\textbf{Burning bijection anchored at the origin.}
The burning process will consist of two Phases.

\emph{Phase I.} We burn all vertices we can \emph{without} burning
the origin. That is, we define $B^{(I)}_0 = \{ s \}$, $U^{(I)}_0 = V_n$, and
for $t \ge 1$ we set:
\eqnsplst
{ B^{(I)}_{t}
  &:= \left\{ x \in U^{(I)}_{t-1} \setminus \{ o \} : 
     \eta(x) \ge \deg_{U^{(I)}_{t-1}}(x) \right\} \\
  U^{(I)}_t
  &:= U^{(I)}_{t-1} \setminus B^{(I)}_t. }
At some finite time no more vertices can be burnt, that is, $B^{(I)}_{J} = \es$
for some $1 \le J < \infty$.

\emph{Phase II.} Burn all the remaining vertices in the usual way.
That is, we start with $B^{(II)}_0 = \cup_{j \ge 0} B^{(I)}_j$, 
$U^{(II)}_0 = \cap_{j \ge 0} U^{(I)}_j$, and for $t \ge 1$ set 
\eqnsplst
{ B^{(II)}_{t}
  &:= \left\{ x \in U^{(II)}_{t-1} : 
     \eta(x) \ge \deg_{U^{(II)}_{t-1}}(x) \right\} \\
  U^{(II)}_t
  &:= U^{(II)}_{t-1} \setminus B^{(II)}_t. }

It is not difficult to see that for all $\eta \in \cR_G$, 
all vertices burn eventually.

Now build a bijection $\varphi_o : \cR_G \to \cT_G$ based on the 
above burning rule, similarly to what we did in Section \ref{ssec:burning}.
That is, if $x \in B^{(I)}_t$ for some $t \ge 1$, draw an oriented edge from $x$
to one of its neighbours in $B^{(I)}_{t-1}$. If $x \in B^{(II)}_t$ for 
some $t \ge 1$, draw an edge from $x$ to one of its neighbours 
in $B^{(II)}_{t-1}$. In both cases, break ties as in the usual bijection, 
if necessary.

\medbreak

The following two claims are not difficult to verify, and are left as 
exercises.

\begin{exercise}
\label{ex:desc}
Put 
\eqnst
{ W_n
  = W_n (\eta) 
  = \{ \text{vertices that did not burn in Phase I} \} 
  = U^{(II)}_0. }
Then $W_n = \{ \text{descendants of $o$ in $\tau = \varphi_o(\eta)$} \}$. 
Here a vertex $w$ is called a descendant of the vertex $v$, if $v$
lies on the unique path between $w$ and $s$ in the tree $\tau$.
See \cite{JW12}. 
\end{exercise}

\begin{exercise}
\label{ex:cond-height}
Under the measure $\nu_n$ and conditional on the event $\deg_{W_n} (o) = i$, 
the random variable $\eta(o)$ is uniformly distributed on $\{ i, i+1, \dots, 2d-1 \}$.
\emph{Hint:} Condition further on the entire set $W_n$, and consider
the possible values of $\eta(o)$ in relation to the outgoing edge
from $o$ in $\varphi_o(\eta)$.
See \cite{LP13}*{Lemma 4}.
\end{exercise}

The $p(i)$'s can be rephrased in terms of the quantities
\eqnst
{ q(i)
  := \lim_{n \to \infty} \nu_n [ \deg_{W_n}(0) = i ], \quad i = 0, 1, \dots, 2d - 1. }
Existence of the limit follows, for example, from results presented in 
Section \ref{sec:measures}. Due to Exercise \ref{ex:cond-height}, we have
\eqn{e:p-q}
{ p(i) 
  = \sum_{j=0}^i \frac{1}{2d-j} q(j). } 
Linearity of expectation and Wilson's
algorithm yield
\eqnsplst
{ \sum_{i=0}^{2d-1} i q(i)
  &= \lim_{n \to \infty} \sum_{x \sim o} 
    \nu_n [ x \in W_n ] \\
  &= \lim_{n \to \infty} \sum_{x \sim o} 
    \Pr [ \text{LERW from $x$ to $s$ in $G_n$ visits $o$} ]. }
By the definition of the infinite LERW and translation invariance
the right hand side equals
\eqnsplst
{ &\sum_{x \sim o} \Pr [ \text{infinite LERW started from $x$ visits $o$} ] \\
  &\qquad\quad = \sum_{x \sim o} \Pr [ 
      \text{infinite LERW started from $o$ visits $-x$} ] \\
  &\qquad\quad = \xi. }
The relation \eqref{e:p-q} now yields $\zeta = d + \frac{\xi - 1}{2}$.

\medbreak

We are now ready to present the main ideas of Priezzhev's computation 
of $p(1)$ and $p(2)$. (We have seen in Section \ref{ssec:height-0} that 
$p(0) = \frac{1}{4} q(0) = \frac{2}{\pi^2} - \frac{4}{\pi^3}$.)
For the remainder of Section \ref{ssec:height-123}, we restrict to $d = 2$. 
One of our concerns will be to supply explicit error bounds that allow 
one to pass to the limit $n \to \infty$ in the computations. These are 
not provided in the physics literature, and we believe that such 
estimates may be useful in further work on related questions, 
and therefore would benefit a reader who is not yet familiar with 
the details of the work of physicists. Attention is also due to the 
fact that Priezzhev's integrals \cite{Pr94}*{Eqn.~(6)} include 
logarithmically divergent singularities, and do not exist as Lebesgue 
integrals; see our Remark \ref{rem:not-integrable}. Implicit in Priezzhev's 
formula is to apply a regularization that allows divergent singularities 
to cancel. We provide a suitable regularization in 
Propositions \ref{prop:limn} and \ref{prop:summation}.

\subsubsection{Decomposition of $q(1)$ into three terms}

Due to \eqref{e:p-q}, it is enough to find $q(1)$ and $q(2)$. We restrict 
to the computation of $q(1)$, as the computations for $q(2)$ follow similar 
ideas; see \cites{Pr94,JPR06}.

Let $q_n(1) = \nu_n [ \deg_{W_n}(o) = 1 ]$. We will work in the (large) 
finite graph $G_n$. Let $j_1, j_2, j_3, j_4$ be the south, west, north, east 
neighbours of the origin $o$, respectively.\footnote{We have tried to keep 
the notation consistent with that of \cite{Pr94} as much as possible.} 
Due to symmetry, we have
\eqn{e:x1-fixed}
{ q_n(1)
  = 4\, \nu_n [ \deg_{W_n}(o) = 1,\, j_1 \in W_n,\, j_2, j_3, j_4 \not\in W_n ]. }
It will be useful to regard spanning trees of $G_n$ as being oriented towards $s$.
Then specifying a spanning tree is equivalent to specifying an acyclic rotor
configuration $\rho$ on $V_n$, and we are required to count certain
acyclic rotor configurations.\footnote{\emph{Note:} Here we are using 
the bijection anchored at $o$ introduced in Section \ref{sssec:looping}, 
and \emph{not} the sandpile group action on acyclic rotors of 
Section \ref{ssec:rotor-router}.}
The event on the right hand side of \eqref{e:x1-fixed}
is equivalent to the event that the rotor at $j_1$ is pointing to $o$, and 
there is no directed path from any of $j_2, j_3, j_4$ to $j_1$.
Using the idea of Exercise \ref{ex:cond-height}, we can fix the rotor at $o$
to be pointing to $j_2$, say, and introduce a factor $3$. That is:
\eqnst
{ q_n(1)
  = \frac{12}{\det(\Delta'_n)} \left| \left\{ \rho: 
    \parbox{6cm}{$\rho$ acyclic, 
    $\rho(j_1) = [j_1,o]$, $\rho(o) = [o,j_2]$, $\head(\rho(j_3)) \not= o$,
    $\head(\rho(j_4)) \not= o$, no oriented path from $j_3$ and $j_4$ to $j_1$} 
    \right\} \right|. }
Due to planarity, it is in fact enough to require that there be no 
oriented path from $j_4$ to $j_1$. This is because if $j_3$ had such 
a path, so would $j_4$, due to the fact that $j_2$ has an oriented 
path to the sink. Hence
\eqn{e:acyclic-formula}
{ q_n(1)
  = \frac{12}{\det(\Delta'_n)} \left| \left\{ \rho: 
    \parbox{6cm}{$\rho$ acyclic, 
    $\rho(j_1) = [j_1,o]$, $\rho(o) = [o,j_2]$, $\head(\rho(j_3)) \not= o$,
    $\head(\rho(j_4)) \not= o$,
    no oriented path from $j_4$ to $j_1$} \right\} \right|. }
The non-local constraint that there be no oriented path from 
$j_4$ to $j_1$ amounts to requiring that if a \emph{second}
rotor were introduced at $o$, pointing to $j_4$, then the
resulting configuration would still be acyclic. For short 
let $e = [o,j_2]$, $f = [o,j_4]$, $h = [j_1,o]$, and put
\eqnst
{ \cT_0
  = \left\{ \rho_0 : \parbox{9.5cm}{$\rho_0$ an acyclic rotor 
    configuration on $V_n \setminus \{ o \}$, 
    $\rho_0(j_1) = h$, $\head(\rho_0(j_2)) \not= o$, 
    $\head(\rho_0(j_3)) \not= o$, $\head(\rho_0(j_4)) \not= o$} \right\}. }
Put 
\eqnsplst
{ \cT_e 
  &:= \{ \rho_0 \in \cT_0 : \text{$\rho_0 \cup \{ e \}$ is acyclic} \} \\
  \cT_f
  &:= \{ \rho_0 \in \cT_0 : \text{$\rho_0 \cup \{ f \}$ is acyclic} \}. }
Then $|\cT_e \cap \cT_f|$ counts the number of elements of the set in the
right hand side of \eqref{e:acyclic-formula}, that we write as:
\eqn{e:decomp-AcB}
{ |\cT_e \cap \cT_f|
  = |\cT_e| - |\cT_f^c| + |\cT_e^c \cap \cT_f^c|. }

\subsubsection{An extension of the Matrix-tree theorem}

In order to get formulas for the three terms in \eqref{e:decomp-AcB},
we are going to use the theorem below that states variations 
on the Matrix-Tree Theorem for directed graphs \cite{Bbook}*{Theorem II.14}. 
Let $G = (V \cup \{ s \}, E)$ be a \emph{directed} graph,
with $-\Delta_{xy} = a_{xy}$ the number of directed edges from $x$ to $y$, and 
$\Delta_{xx} = \outdeg(x) := \sum_{y \in V \cup \{s\}} a_{xy}$. 
We assume that $\Delta_{sx} = 0$ for 
all $x \in V$. From now on, but in this section only, we are going to call an oriented 
cycle of $G$ an \emph{oriented loop} (to distinguish from permutation 
cycles in the proof below). In order to state the theorem,
we need some notation. Let $N_0$ denote the number of rotor configurations 
on $V$ with no oriented loop (acyclic rotor configurations). Given a directed 
edge $h$, let $N_1(h)$ denote the number of rotor configurations that contain 
precisely one oriented loop, with the edge $h$ contained in this loop. Let 
\eqnst
{ \widetilde{\Delta}_{xy}^h
  = \begin{cases}
    \Delta_{xy} & \text{if $[x,y] \not= h$;} \\
    - \omega    & \text{if $[x,y] = h$;}
    \end{cases} }
where $\omega$ is a real parameter.
Similarly, given oriented egdes $f_1, f_2, f_3$ of $G$, 
let $N_i(f_1,f_2,f_3)$, $i = 1, 2, 3$ respectively,
denote the number of rotor configurations that contain precisely
$i$ oriented loops, respectively, in such a way that each loop contains at least 
one of $f_1, f_2, f_3$, and each of $f_1, f_2, f_3$ is contained 
in at least one of the loops. Let
\eqnst
{ \widetilde{\Delta}_{xy}^{f_1,f_2,f_3}
  = \begin{cases}
    \Delta_{xy} & \text{if $[x,y] \not= f_1, f_2, f_3$;} \\
    - \omega    & \text{if $[x,y] \in \{ f_1, f_2, f_3 \}$}.
    \end{cases} }
As before, let $\Delta'$ denote the matrix obtained from $\Delta$ by restricting
the indices to $V \times V$.

\begin{theorem}
\label{thm:MT-Pr}
(Priezzhev \cite{Pr94})
We have:
\eqnsplst
{ \det(\Delta') 
  &= N_0 \\
  \lim_{\omega \to \infty} \frac{1}{\omega} 
  \det( (\widetilde{\Delta}^h)' ) 
  &= - N_1(h) \\
  \lim_{\omega \to \infty} \frac{1}{\omega^3} 
  \det( (\widetilde{\Delta}^{f_1,f_2,f_3})' ) 
  &= - N_1(f_1,f_2,f_3) + N_2(f_1,f_2,f_3) - N_3(f_1,f_2,f_3). }
\end{theorem}

\begin{proof}
Expand $\det(\Delta')$ as a sum over permutations of $V$, and 
for each permutation in the sum, consider its decomposition into 
cyclic permutations. Define the \emph{weight} of a permutation cycle 
$(x_1,\dots,x_k)$ of length $k \ge 2$ to be $\prod_{i=1}^k a_{x_i,x_{i+1}}$,
and the weight of a ``trivial'' permutation cycle $(x)$ of length $1$ 
to be $\Delta_{xx}$. Hence we have:
\eqnst
{ \det(\Delta')
  = \sum_{\text{permutations}} (-1)^{\# \text{non-trivial perm.~cycles}}
    \prod_{\text{perm.~cycles}} \weight(\text{perm.~cycle}). }
Note that a non-trivial permutation cycle of $k$ edges, $k \ge 2$, brings a sign 
$(-1)^k$ due to $k$ factors of $-a_{x_i,x_{i+1}}$, and therefore the factor 
$(-1)^{\# \text{non-trivial perm.~cycles}}$ ensures the correct sign
for the signature of the permutation. The weight of a non-trivial cycle
counts the number of oriented loops with the same vertex set.
Let us group terms according to the number of non-trivial loops 
and write $\Gamma$ for the set of all oriented loops in $G$. This yields:
\eqnsplst
{ \det(\Delta')
  &= \prod_{x \in V} \outdeg(x)
    - \sum_{\gamma_1 \in \Gamma} \ \prod_{x \in V \setminus \gamma_1} \outdeg(x) \\
  &\qquad + \sum_{\substack{\gamma_1, \gamma_2 \in \Gamma \\ \gamma_1 \not= \gamma_2}}
      \ \prod_{x \in V \setminus (\gamma_1 \cup \gamma_2)} \outdeg(x)
    - \dots. }
The first term counts the number of all rotor configurations on $V$. The
summand in the second term is the number of all rotor configurations that 
contain the oriented loop $\gamma_1$. The summand in the third term counts the 
number of rotor configurations that contain both loops $\gamma_1$ and $\gamma_2$; 
and so on. It follows from the inclusion-exclusion principle that the alternating 
sum counts precisely the number of rotor configurations with no loops. 
This proves the first statement.

Consider now the same expansion for the modified matrix $(\widetilde{\Delta}^h)'$.
Due to the factor $\frac{1}{\omega}$, the only terms that remain are the ones
where one of the oriented loops contains the edge $h$. Note that for each term
there is at most one such loop. Grouping terms according to what this loop is:
\eqnsplst
{ \lim_{\omega \to \infty} \frac{1}{\omega} \det( (\widetilde{\Delta}^h)' )
  &= - \sum_{\substack{\gamma_1 \in \Gamma : \\ h \in \gamma_1}}
    \Bigg[ \prod_{x \in V \setminus \gamma_1} \outdeg(x) 
   - \sum_{\substack{\gamma_2 \in \Gamma : \\ \gamma_2 \not= \gamma_1}}
    \ \prod_{x \in V \setminus (\gamma_1 \cup \gamma_2)} \outdeg(x) \\
  &\qquad\qquad\quad + \sum_{\substack{\gamma_2 \not= \gamma_3 \in \Gamma : \\ \gamma_2, \gamma_3 \not= \gamma_1}} 
    \ \prod_{x \in V \setminus (\gamma_1 \cup \gamma_2 \cup \gamma_3)} \outdeg(x) 
    - \dots \Bigg]. }
For each fixed $\gamma_1 \ni h$, the expression inside the square brackets
is an inclusion-exclusion formula for the number of rotor configurations
that contain $\gamma_1$ but no other oriented loop. Hence the 
second statement follows.

The third statement can be proved similarly to the second. This time the only
terms that remain are the ones where there is a set of one, two or three loops
that together contain $f_1, f_2, f_3$. Grouping terms according to what 
these loops are, we get a sum of inclusion-exclusion formulas yielding the terms
$(-1)^i N_i(f_1,f_2,f_3)$, $i = 1, 2, 3$.
\end{proof}

\subsubsection{The term $|\cT_e|$}

Let us return to the three terms in \eqref{e:decomp-AcB}. The first term
$|\cT_e|$ only involves local restrictions: certain rotor directions are 
forced or forbidden. Let us denote by $j''_1, j''_2, j''_4$ the south,
west, east neighbours of $j_1$, respectively. The rotor $[o,j_2] = e$ can be
forced by deleting the oriented edges $[o,j_1]$, $[o,j_3]$, $[o,j_4]$
from the graph. The rotor $[j_1,o] = h$ can be forced by deleting the oriented 
edges $[j_1,j''_1]$, $[j_1,j''_2]$, $[j_1,j''_4]$ from the graph. The 
requirements $\head(\rho_0(j_3)) \not= o$ and $\head(\rho_0(j_4)) \not= o$ 
can be achieved by deleting the oriented edges $[j_3,o]$ and $[j_4,o]$
from the graph (and note that the requirement $\head(\rho_0(j_2)) \not= o$
becomes superfluous due to acyclicity). It follows that we can apply the first 
statement of Theorem \ref{thm:MT-Pr} to the matrix 
$(\Delta^{(1)}_n)' = \Delta'_n + \delta^{(1)}$, 
where the only nonzero entries of $\delta^{(1)}$ are:
\eqnst
{ \delta^{(1)}
  = \begin{blockarray}{cccccccc}
  o & j_1 & j_3 & j_4 & j''_1 & j''_2 & j''_4 &  \\
  \begin{block}{(ccccccc)c}
 -3 &   1 &  1  &   1 &  0  &  0  &  0  &  0 \\
  0 &  -3 &  0  &   0 &  1  &  1  &  1  & j_1 \\
  1 &   0 &  -1 &   0 &  0  &  0  &  0  & j_3 \\
  1 &   0 &   0 &  -1 &  0  &  0  &  0  & j_4 \\
  \end{block}
  \end{blockarray} }
Using explicit values of the potential kernel (see~\eqref{e:pot-kern}), 
we get
\eqnst
{ \lim_{n \to \infty} \frac{12 |\cT_e|}{\det(\Delta'_n)}
  = \lim_{n \to \infty} 12\, \det ( I + \delta^{(1)} G_n )
  = \frac{6}{\pi} - \frac{30}{\pi^2} + \frac{48}{\pi^3}. }

\subsubsection{The term $|\cT_f^c|$}

The second term $|\cT_f^c|$ in \eqref{e:decomp-AcB} 
involves the non-local restriction that 
$f$ is contained in a loop. Necessarily, this loop
ends with the edge $h$. We are going to force the loop by
giving $h$ the weight $-\omega$, and deleting the oriented
edges $[o,j_1], [o,j_2], [o,j_3]$. We also delete
$[j_2,o]$, $[j_3,o]$. (Note that this time the 
requirement $\head(\rho_0(j_4)) \not= o$ is superfluous.)
Since $\omega \to \infty$, the rest of row $j_1$ of 
the matrix is immaterial.
Hence we apply the second statement of 
Theorem \ref{thm:MT-Pr} to the matrix 
$(\Delta^{(2)}_n)' = (\widetilde{\Delta}^f)' = \Delta'_n + \delta^{(2)}$, 
where now
\eqnst
{ \delta^{(2)}
  = \begin{blockarray}{ccccc}
  o      & j_1 & j_2 & j_3 &   \\
  \begin{block}{(cccc)c}
 -3      &   1 &  1  &  1  &  o \\
 -\omega &   0 &  0  &  0  & j_1 \\
  1      &   0 &  -1 &  0  & j_2 \\
  1      &   0 &   0 & -1  & j_3 \\
  \end{block}
  \end{blockarray} }
Since this time row $j_1$ of $\delta^{(2)}$ does not sum to $0$, a divergent
term of order $\log n$ arises. (This reflects the fact that the 
number of configurations containing a cycle is much larger than
the number of acyclic ones.) Since we are evaluating a probability,
the divergence will have to be cancelled by a term we get for 
$|\cT_e^c \cap \cT_f^c|$. In order to deal with the divergent
terms, we need the following lemma. 

\begin{lemma}
For $K \ge 1$ and $z, w \in \Z^2$ with $|z|, |w| \le K$, as 
$n \to \infty$, we have
\eqn{e:G_n(z,w)}
{ G_n(z,w) 
  = G_n(o,o) - A(w-z) + O_K \left( \frac{\log n}{n} \right), }
where the constant implied by $O_K$ depends on $K$.
\end{lemma}

We do not prove this, as this is incorporated in Lemma \ref{lem:estimates}
to come.


Replacing each term $G_n(z,w)$ in the matrix $I + \delta^{(2)} G_n$ by the
expression on the right hand side of \eqref{e:G_n(z,w)} and
taking into account that $G_n(o,o) = O(\log n)$ (see Lemma \ref{lem:estimates}) 
we get:
\eqnspl{e:cTf}
{ \frac{-12 |\cT_f^c|}{\det(\Delta'_n)}
  &= 12\, \lim_{\omega \to \infty} \frac{1}{\omega} \det \left( I + \delta^{(2)} G_n \right) \\
  &= \frac{3}{\pi^2} - G_n(o,o)\, 12\, \left( \frac{2}{\pi^2} - \frac{4}{\pi^3} \right)
     + O \left( \frac{\log^2 n}{n} \right). }

\subsubsection{Priezzhev's ``bridge trick'' for the term $|\cT_e^c \cap \cT_f^c|$}
\label{sssec:bridge}

We are left to calculate the term $|\cT_e^c \cap \cT_f^c|$. 
Let $\rho_0 \in \cT_e^c \cap \cT_f^c$, and let $\rho$ stand for 
the set of edges $\rho = \rho_0 \cup \{ e \} \cup \{ f \}$.
There is a unique vertex $i_1 \in V_n \setminus \{ o \}$ 
such that $\rho$ contains three oriented paths:\\
(i) a path from $o$ to $i_1$ starting with $e$;\\
(ii) a path from $o$ to $i_1$ starting with $f$;\\
(iii) a path from $i_1$ to $o$ ending with $h$.\\
Moreover, the three paths are vertex-disjoint apart from
the vertices $o$ and $i_1$. We will call $i_1$ the \emph{meeting point}.
We are going to count configurations in $\cT_e^c \cap \cT_f^c$
separately for each fixed value $i_1$ of the meeting point.

The idea is to add three ``bridge'' edges between $j_1, j_2, j_4$ and 
three neighbours of $i_1$, and force the bridge edges to be in loops,
via the third statement of Theorem \ref{thm:MT-Pr}.
Then the existence of the loops gives the required paths,
apart from possible flips in orientation of some of the paths.
We would need to sum over all possible locations of $i_1$
and all possible choices of three neighbours of $i_1$.
It turns out that symmetry considerations allow one to reduce 
the amount of calculations, and to count only two types of 
configurations according to the pattern of edges near $i_1$.
In order to define these patterns, let $w_1, w_3, w_4$ be the 
south, north, east neighbours of $i_1$, respectively.
Let $G^{L,i_1}_n$ be the graph obtained from $G_n$ by 
removing $[j_3,o]$ and adding three ``bridges'':
\eqnst
{ f^L_1 = [j_1,i_1], \qquad
  f^L_2 = [j_2,w_4], \qquad
  f^L_3 = [j_4,w_3]. }
Following Priezzhev's notation \cite{Pr94}, the symbol $L$ indicates that 
the vertices $w_3, i_1, w_4$ 
form an $L$-shape. Let $G^{\Gamma,i_1}_n$ be the graph obtained
from $G_n$ by removing $[j_3,o]$ and adding the three bridges:
\eqnst
{ f^\Gamma_1 = [j_1,i_1], \qquad
  f^\Gamma_2 = [j_2,w_1], \qquad
  f^\Gamma_3 = [j_4,w_4]. }
The symbol $\Gamma$ indicates the $\Gamma$-shape formed by the
vertices $w_1, i_1, w_4$.
If $i_1 = (k,l)$, we denote $i_1' = (-k,l)$.
The computation is based on the following lemma.

\begin{lemma}
\label{lem:N_i's}
There is a finite (explicit) set $P \subset \Z^2$, such that 
whenever $i_1, i_1' \in V_{n-1} \setminus P$, the following holds.\\
(i) We have $N_2(f^L_1, f^L_2, f^L_3; z) = 0 = 
N_2(f^\Gamma_1, f^\Gamma_2, f^{\Gamma}_3 ; z)$ for $z = i_1, i_1'$.\\
(ii) We have
\eqnsplst
{ &\sum_{z \in \{ i_1, i_1' \}} \ \sum_{i = 1, 3}
    \left[ N_i(f^L_1, f^L_2, f^L_3 ; z) 
    + N_i(f^\Gamma_1, f^\Gamma_2, f^\Gamma_3 ; z) \right] \\
  &\qquad\qquad\qquad = |\cT_e^c \cap \cT_f^c 
    \cap \{ \text{meeting point equals $i_1$ or $i_1'$} \}|. }
\end{lemma}

\begin{proof}
(i) This follows from planarity, as can be checked case-by-case. 
(This is the point in the computation where $d = 2$ is used in a 
crucial way.)

(ii) Consider first a configuration counted in 
$N_3(f^L_1, f^L_2, f^L_3; i_1)$. Remove the bridges.
There are three vertex-disjoint oriented paths
$i_1 \to o$, $w_4 \to j_2$ and $w_3 \to j_4$.  
Reverse the orientations of the paths arriving at 
$j_2$ and $j_4$, respectively. This leaves a rotor
configuration with no rotor specified at $o$, $w_3$
and $w_4$. Adding the rotors $[w_3,i_1], [w_4,i_1]$ we get
a configuration in $\cT_e^c \cap \cT_f^c$ such that 
the meeting point is $i_1$. The operations we performed are
one-to-one between $N_3(f^L_1, f^L_2, f^L_3; i_1)$ and the
configurations in the image. 

We can perform similar steps 
for $N_1(f^L_1, f^L_2, f^L_3; i_1)$: remove the bridges,
and reverse the orientation of the paths arriving 
at $j_2$ and $j_4$. The paths can occur in two distinct
ways. One is obtained when we started with 
$i_1 \to j_2$, $w_4 \to j_4$, $w_3 \to o$,
in which case, after we reversed orientations,
there is no rotor specified at $i_1$ and 
$w_4$. We set these rotors as $[i_1,w_3]$ and $[w_4,i_1]$,
yielding a configuration in $\cT_e^c \cap \cT_f^c$.
The other possibility is that we started with 
$i_1 \to j_4$, $w_3 \to j_2$, $w_4 \to o$, in which case
rotors will be missing at $i_1$ and $w_3$. We set these
to be $[i_1,w_4]$ and $[w_3,i_1]$ to get a configuration 
in $\cT_e^c \cap \cT_f^c$. 

Observe that the configurations constructed from $N_1$ 
are distinct from the ones arising from $N_3$. 
Let us add now the `$\Gamma$-configurations', that is those 
arising from
$N_3(f^\Gamma_1, f^\Gamma_2, f^\Gamma_3; i_1)$ and
$N_1(f^\Gamma_1, f^\Gamma_2, f^\Gamma_3; i_1)$. There are four 
possibilities for the pattern of three edges incident 
with $i_1$ involved in the three paths. Configurations 
where the three edges only contain either the $L$- or the 
$\Gamma$-shape have been counted exactly once.
Configurations where the three edges contain both 
the $L$- and the $\Gamma$-shape have been counted 
twice. The mirror images of these configurations
(under $i_1 \leftrightarrow i_1'$) on the other hand,
whose number is the same, are \emph{not} counted in 
the corresponding terms $N_1(\cdot, \cdot, \cdot ; i_1')$,
$N_3(\cdot, \cdot, \cdot ; i_1')$. Hence adding together
the contributions for $i_1$ and $i_1'$ restores the
balance, and implies the statement.

The symmetry argument breaks down if $i_1 \in V_n \setminus V_{n-1}$.
The path arguments may break down if $i_1$ or $i_1'$ is too close to $o$, 
so there is a set of exceptions $P$.
\end{proof}

\begin{remark}
When we take the limit $n \to \infty$, we are going to 
sum over all $i_1 \in \Z^2$ to have a workable expression 
as a Fourier integral. Then one needs to correct for the 
exceptions $P$ individually. These can be handled with 
the ideas we used for $|\cT_e|$ and $|\cT_f^c|$; see \cite{Pr94}. 
The divergent term in \eqref{e:cTf} is cancelled by similar
divergent terms in the contributions of the exceptional 
points in $P$. It also follows from the estimates in 
Lemma \ref{lem:Aasymp} that the contribution of the boundary terms 
$i_1, i_1' \in V_n \setminus V_{n-1}$ is negligible in the limit.
\end{remark}

\subsubsection{Summation formulas for the $N_1$ and $N_3$ terms}

In evaluating the $N_1$ and $N_3$ terms, the bridges 
receive weight $-\omega$. The necessary modifications of the 
graph are encoded in the matrices:
\eqnst
{ \Delta_{i_1}(L)
  = \Delta'_n + \delta_{i_1}(L), \qquad\qquad
  \Delta_{i_1}(\Gamma)
  = \Delta'_n + \delta_{i_1}(\Gamma), }
where
\eqnsplst
{ \delta_{i_1}(L)
  &= \begin{blockarray}{cccccc}
 j_3 &  o     &  i_1    & a = w_4 & b = w_3 &  \\
  \begin{block}{(ccccc)c}
 -1 &   1     &   0     &  0      &  0      & j_3 \\
  0 &   0     & -\omega &  0      &  0      &  o \\
  0 &   0     &   0     & -\omega &  0      & j_2 \\
  0 &   0     &   0     &  0      & -\omega & j_4 \\
  \end{block}
  \end{blockarray} \\
  \delta_{i_1}(\Gamma)
  &= \begin{blockarray}{cccccc}
 j_3 &  o     &  i_1      & a = w_1     & b = w_4     &  \\
  \begin{block}{(ccccc)c}
 -1 &   1     &   0     &  0      &  0      & j_3 \\
  0 &   0     & -\omega &  0      &  0      &  o \\
  0 &   0     &   0     & -\omega &  0      & j_2 \\
  0 &   0     &   0     &  0      & -\omega & j_4 \\
  \end{block}
  \end{blockarray} }

Due to Theorem \ref{thm:MT-Pr} and Lemma \ref{lem:N_i's}(i) we have:
\eqnspl{e:N_i's-det-form}
{ &N^{L,i_1}_1 + N^{L,i_1}_3
  = \lim_{\omega \to \infty} \frac{-1}{\omega^3} 
    \frac{\det (\Delta_{i_1}(L))}{\det(\Delta'_n)} 
  = \lim_{\omega \to \infty} \frac{-1}{\omega^3} \det ( I + \delta_{i_1}(L) G_n ) \\
  &= \begin{vmatrix}
    1 + G_n(o,j_3)       & G_n(o,o)            & G_n(o,j_2)           & G_n(o,j_4)            \\
       \ -G_n(j_3,j_3) & \quad\ -G_n(j_3,o) & \quad\ -G_n(j_3,j_2) &  \quad\ -G_n(j_3,j_4) \\[1ex]   
    G_n(i_1,j_3)                & G_n(i_1,o)          & G_n(i_1,j_2)            & G_n(i_1,j_4) \\[1ex]
    G_n(w_4,j_3)                & G_n(w_4,o)          & G_n(w_4,j_2)            & G_n(w_4,j_4) \\[1ex]
    G_n(w_3,j_3)                & G_n(w_3,o)          & G_n(w_3,j_2)            & G_n(w_3,j_4)
    \end{vmatrix} \\
  &=: \det(M_L). }
We similarly get
\eqnspl{e:N_i's-det-form2}
{ &N^{\Gamma,i_1}_1 + N^{\Gamma,i_1}_3 \\
  &= \begin{vmatrix}
    1 + G_n(o,j_3)       & G_n(o,o)          & G_n(o,j_2)           & G_n(o,j_4)              \\
    \ -G_n(j_3,j_3) & \quad\ -G_n(j_3,o) &  \quad\ -G_n(j_3,j_2) & \quad\ -G_n(j_3,j_4)    \\[1ex]
    G_n(i_1,j_3)                & G_n(i_1,o)          & G_n(i_1,j_2)            & G_n(i_1,j_4) \\[1ex]
    G_n(w_4,j_3)                & G_n(w_4,o)          & G_n(w_4,j_2)            & G_n(w_4,j_4) \\[1ex]
    -G_n(w_1,j_3)               & -G_n(w_1,o)         & -G_n(w_1,j_2)           & -G_n(w_1,j_4)
    \end{vmatrix} \\
  &=: - \det(M_\Gamma). }
In order to deal with divergences as $n \to \infty$, we regularize 
by replacing $G_n$ in rows 2--4 of the matrices $M_L$ and $M_\Gamma$ 
with the Green's function of the geometrically killed random walk:
\eqnst
{ G_n(z,w; r)
  := \frac{1}{4} \sum_{m = 0}^\infty r^m \Pr^z [ S(m) = w,\, \tau_{V_n^c} > m ], \quad 0 < r \le 1, }
where $\tau_{V_n^c}$ is the hitting time of $V_n$.
Let $M_{L,r}$ and $M_{\Gamma,r}$ be the matrices obtained this way.
We also let 
\eqnsplst
{ A(z,w;r)
  &:= \frac{1}{4} \lim_{N \to \infty} \sum_{m=0}^N r^m \left( \Pr^z [ S(m) = w ] - \Pr^z [ S(m) = z ] \right), \\
  &\qquad z, w \in \Z^2,\, 0 < r \le 1, }
and
\eqnst
{ G_{k,l}(r)
  := \frac{1}{4} \sum_{m = 0}^\infty r^m \Pr^o [ S(m) = (k,l) ], \quad (k,l) \in \Z^2,\, 0 < r < 1. }
The following two propositions, that we prove in 
Sections \ref{sssec:estimates}--\ref{sssec:summation},
state summation formulas for the $N_1$ and $N_3$ terms 
in the limit $n \to \infty$. These two propositions form the remaining 
part of the computation of $|\cT_e^c \cap \cT_f^c|$.

\begin{proposition}
\label{prop:limn}
We have
\eqnsplst
{ \lim_{n \to \infty} \sum_{i_1 \in V_n} \left[ \det(M_L) - \det(M_\Gamma) \right]
  = \lim_{r \uparrow 1} \sum_{(k,l) \in \Z^2} \det(C_{k,l}(r)), }
where
\eqnspl{e:Mr}
{ &C_{k,l}(r)  
  := \\
  &\begin{pmatrix}
    \frac{3}{4}       & \frac{1}{4}       & \frac{1}{\pi} - \frac{1}{4} & \frac{1}{\pi} - \frac{1}{4} \\[1ex]
    G_{k,l-1}         & G_{k,l}           & G_{k+1,l}                   & G_{k-1,l}   \\[1ex]
    G_{k+1,l-1}       & G_{k+1,l}         & G_{k+2,l}                   & G_{k,l}    \\[1ex]
    G_{k,l}           & G_{k,l+1}         & G_{k+1,l+1}        & G_{k-1,l+1}  \\
    \quad\quad -G_{k,l-2}  &  \quad\ -G_{k,l-1} & \quad\ -G_{k+1,l-1} & \quad\ -G_{k-1,l-1} 
    \end{pmatrix}, }
with each $G$-entry evaluated at $r$.
\end{proposition}

\begin{proposition}
\label{prop:summation}
We have
\eqnst
{ \sum_{(k,l) \in \Z^2} \det(C_{k,l}(r))
  = \frac{1}{32 \pi^4} \iint \iint
    \frac{i \sin \beta_1 \, \det(M_1)\, d\alpha_1\, d\beta_1\, d\alpha_2\, d\beta_2}{D_r(\alpha_1,\beta_1) 
    D_r(\alpha_2, \beta_2) D_r(\alpha_1 + \alpha_2,\beta_1 + \beta_2)}, }
where 
\eqnst
{ M_1
  := \begin{pmatrix}
    \frac{3}{4}         & \frac{1}{4} & \frac{1}{\pi} - \frac{1}{4} & \frac{1}{\pi} - \frac{1}{4} \\
    e^{i(\beta_1 + \beta_2)}  & 1           & e^{-i(\alpha_1 + \alpha_2)}       & e^{i(\alpha_1 + \alpha_2)} \\
    e^{i(\alpha_2 - \beta_2)} & e^{i \alpha_2} & e^{2i \alpha_2}                & 1                    \\
    e^{-i\beta_1}          & 1           & e^{i \alpha_1}                & e^{-i \alpha_1}
    \end{pmatrix}, }
and $D_r(\alpha,\beta) = 2 - r (\cos \alpha + \cos \beta)$.  
\end{proposition}

\begin{remark}
\label{rem:not-integrable}
When $r = 1$, the integral does not exist as a Lebesgue integral. In order to 
see this, consider the region of integration where $(\alpha_2,\beta_2)$ is in
a small neighbourhood of $(0,0)$, and $(\alpha_1, \beta_1)$ is in a small
neighbourhood of $(\pi/4,\pi/4)$, say. Subtract the last column from 
the other columns, and expand the determinant along the third row.
The first three terms each contain a factor that vanishes as
$(\alpha_2,\beta_2) \to (0,0)$, making the singularity due to 
$D_1(\alpha_2,\beta_2)$ integrable. But the last term is 
proportional to $(D_1(\alpha_2,\beta_2))^{-1}$, which is not 
integrable. Proposition \ref{prop:summation} exhibits a delicate 
cancellation taking place.
\end{remark}

\subsubsection{Green function estimates}
\label{sssec:estimates}

For the proofs of Propositions \ref{prop:limn} and \ref{prop:summation},
we are going to need some estimates on Green functions. 

\begin{lemma}
\label{lem:Aasymp} \ \\
(i) There exists a constant $K$ such that 
\eqnst
{ A(z,w;1)
  = \frac{1}{2 \pi} \log |w - z| + K + O \left( \frac{1}{|w - z|^2} \right), 
    \quad \text{$|w - z| \ge 1$.} }
(ii) Uniformly in $0 < r \le 1$ and $w \in \Z^2$, we have
\eqnst
{ A(z,w;r) - A(o,w;r)
  = O (\log |z|), 
    \quad \text{$|z| \ge 2$.} }
(iii) Uniformly in $0 < r \le 1$ and for $|f| = 1 = |h|$ we have
\eqnsplst
{ \partial^{(1)}_f A(z,o;r)
  &= O \left( \frac{1}{|z|} \right), 
    \quad \text{$|z| \ge 1$;} \\
  \partial^{(2)}_h A(z,o;r)
  &= O \left( \frac{1}{|z|} \right),
     \quad \text{$|z| \ge 1$;} \\
  \partial^{(1)}_f \partial^{(2)}_h A(z,o;r)
  &= O \left( \frac{1}{|z|^2} \right),
     \quad \text{$|z| \ge 1$.} }
\end{lemma}

\begin{proof}
Statement (i) is \cite{LLbook}*{Theorem 4.4.4}.
When $r = 1$, statements (ii) and (iii) follow immediately
from (i). For the case $0 < r < 1$, the proof 
of \cite{LLbook}*{Theorem 4.4.4} can be adapted, and we 
sketch how this can be done. By considering a ``lazy''
random walk (that holds in place with probability $\eps$
on each step), we may replace the simple random walk by 
an aperiodic one. (Indeed, for the lazy walk
$G_{\eps}(z,w;r) = (1 - \eps r)^{-1} G(z,w;r)$; see 
\cite{LLbook}*{(4.17)}.)

Following the proof of \cite{LLbook}*{Theorem 4.4.4}, write 
\eqnsplst
{ A(o,z;r)
  &= \sum_{m \le |z|^2} r^m \p^m(o,o) - \sum_{m \le |z|^2} r^m \p^m(o,z) \\
  &\qquad + \sum_{m > |z|^2} r^m (\p^m(o,o) - \p^m(o,z)). }
Let 
\eqnst
{ B(z;r)
  = \sum_{1 \le m \le |z|^2} \frac{r^m}{m}. }
Then the computations in \cite{LLbook} show that 
\eqnst
{ \sum_{m \le |z|^2} r^m \p^m(o,o)
  = c_1 \, B(z;r) + C + O(|z|^{-2}), }
with $c_1$ and $C$ independent of $z$ and $r$. We also have that 
$\sum_{m \le |z|} r^m \p^m(o,z)$ decays faster than any power of $|z|$
(uniformly in $r$). An application of the local central limit 
theorem yields that 
\eqnst
{ \sum_{|z| < m \le |z|^2} \left[ r^m \p^m(o,z) - r^m \frac{c_1}{m} e^{-|z|^2/m} \right]
  = O ( |z|^{-2} ). }
Also, the proof of \cite{LLbook}*{Lemma 4.3.2} shows that 
\eqnsplst
{ \sum_{|z| < m \le |z|^2} \frac{r^m}{m} e^{-|z|^2/m}
  &= \int_1^\infty \frac{1}{y} \exp \left( - \frac{y}{2} - \frac{\beta |z|^2}{y} \right)
    + O (|z|^{-2}) \\
  &=: I_1(z;r) + O(|z|^{-2}), } 
where $- \beta = \log r \in (-\infty,0)$.
Therefore,
\eqnst
{ \sum_{m \le |z|^2} r^m \left( \p^m(o,o) - \p^m(o,z) \right)
  = c_1 B(z;r) + C + c_1 I_1(z;r) + O(|z|^{-2}). }
A similar computation yields
\eqnsplst
{ \sum_{m > |z|^2} r^m \left( \p^m(o,o) - \p^m(o,z) \right)
  &= c_1 \int_0^1 \frac{1}{y} ( 1 - e^{-y/2} ) 
    e^{- \beta |z|^2 / y }
    + O(|z|^{-2}) \\
  &=: c_1 I_2(z;r) + O(|z|^{-2}). }
Note that $I_1(z;r) = O(1)$ and $I_2(z;r) = O(1)$, uniformly in 
$z$ and $r$. Statement (ii) now follows from
\eqnsplst
{ |A(z,w;r) - A(o,w;r)|
  &= c_1 |B(w-z;r) - B(w;r)| + O(1) \\
  &\le c_1 |B(w-z;1) - B(w;1)| + O(1) \\
  &= O ( \log |z| ). }

In order to prove the statements in (iii), write
\eqnspl{e:partial1A}
{ \partial^{(1)}_f A(z,o;r) 
  &= c_1 [B(z+f;r) - B(z;r)] + c_1 [I_1(z+f;r)-I_1(z;r)] \\
  &\qquad + c_1 [I_2(z+f;r) - I_2(z;r)] + O(|z|^{-2}). }
Using that $|z+f|^2 - |z|^2 = 2 \langle z, f \rangle + 1 = O(|z|)$, 
the first term is $O(|z|^{-1})$. In order to estimate the second term, 
write
\eqnst
{ I_1(z+f;r) - I_1(z;r)
  = \int_1^\infty \frac{1}{y} e^{-y/2} 
    \left( e^{-\beta |z+f|^2/y} - e^{-\beta |z|^2/y} \right). }
We treat the cases $\beta/y \le |z|^{-1}$ and $\beta/y > |z|^{-1}$ separately.
When $\beta/y \le |z|^{-1}$, we have
\eqnspl{e:beta|z|2/y1}
{ \left| e^{-\beta |z+f|^2/y} - e^{-\beta |z|^2/y} \right|
  &= e^{-\beta |z|^2/y} \left| e^{-\beta (2 \langle z, f \rangle + 1)} - 1 \right| \\
  &\le e^{-\beta |z|^2/y} \frac{C \beta |z|}{y} \\
  &= \frac{C'}{|z|} \frac{\beta |z|^2}{y} e^{-\beta |z|^2/y} \\
  &= O(|z|^{-1}). }
When $\beta/y > |z|^{-1}$, we have 
\eqn{e:beta|z|2/y2}
{ e^{-\beta |z+f|^2/y},\, e^{-\beta |z|^2/y}
  = O \left( e^{-c|z|} \right). }
This shows that the second term in \eqref{e:partial1A} is
$O(|z|^{-1})$. Similar considerations apply to the third term in
\eqref{e:partial1A}. The argument for $\partial^{(2)}_h A(z,o;r)$
is identical.

Finally, for the last statement of (iii) we write
\eqnspl{e:partial12A}
{ &\partial^{(1)}_f \partial^{(2)}_h A(z,o;r) \\
  &\qquad = c_1 [B(z+f-h;r) - B(z+f;r) - B(z-h;r) + B(z;r)] \\
  &\qquad\quad + c_1 [I_1(z+f-h;r) - I_1(z+f;r) - I_1(z-h;r) + I_1(z;r)] \\
  &\qquad\quad + c_1 [I_2(z+f-h;r) - I_2(z+f;r) - I_2(z-h;r) + I_2(z;r)] \\
  &\qquad\quad + O(|z|^{-2}). }
In the first term, cancellations take place between the four
summations. The net result is that apart from $O(1)$ terms (that are 
each $O(|z|^{-2})$), there are $O(|z|)$ pairs of terms that come with 
opposite signs. For each pair (treating the cases $r \ge 1 - |z|^{-1}$ and
$r < 1 - |z|^{-1}$ separately), we have the estimate
\eqnst
{ \frac{r^{m_1}}{m_1} - \frac{r^{m_2}}{m_2}
  = O(|z|^{-3}), 
    \quad \text{if $m_1 = |z|^2 + O(|z|)$ and $m_2 = |z|^2 + O(|z|)$.} }
Summing these we get that the first term in \eqref{e:partial12A} is
$O(|z|^{-2})$. For the second term in \eqref{e:partial12A}, 
we argue similarly to \eqref{e:beta|z|2/y1}--\eqref{e:beta|z|2/y2}
(treating the cases $\beta/y \le |z|^{-1}$ and $\beta/y > |z|^{-1}$
separately). This gives
\eqnst
{ e^{-\beta |z+f-h|^2/y} - e^{-\beta |z+f|^2/y} - e^{-\beta|z-h|^2/y} + e^{-\beta |z|^2/y}
  = O (|z|^{-2}), }
and it follows that the second term in \eqref{e:partial12A} is
$O(|z|^{-2})$. The argument for the third term is similar.
This completes the proof.
\end{proof}

Let $\{ S_r(m) \}_{m \ge 0}$ be the random walk killed at a 
$\mathsf{Geometric}(1-r)$ time that is independent of the walk. 
Below we interpret $A(S_r(m),w;r)$ as $0$ after the killing time.

\begin{lemma}
\label{lem:Gn(z,w)}
For all $0 < r \le 1$, $z, w \in V_n$ we have 
\eqnst
{ G_n(z,w;r)
  = \E^z \left[ A(S_r(\tau_{V_n^c}),w;r) \right] - A(z,w;r). } 
\end{lemma}

\begin{proof}
The case $r = 1$ is \cite{LLbook}*{Lemma 4.6.2(b)}. The proof
is similar when $0 < r < 1$. Note that  
$M_m := A(S_r(m),w;r) - \frac{1}{4} \sum_{j = 0}^{m-1} \mathbf{1}_{S_r(j) = w}$
is a martingale. This gives
\eqnst
{ A(z,w;r)
  = \E^z [ A(S_r(N \wedge \tau_{V_n^c}),w;r) ] 
    - \E^z \left[ \sum_{0 \le j < N \wedge \tau_{V_n^c}} \mathbf{1}_{S_r(j) = w} \right]. }
Letting $N \to \infty$ and using bounded and monotone convergence,
respectively, for the two terms we get the statement of the Lemma.
\end{proof}

\begin{lemma}
\label{lem:estimates}
Uniformly in $0 < r \le 1$, $z \in V_n$, $n \ge 1$ and for $|f| = 1 = |h|$, we have
\begin{align}
  G_n(o,o;r)
  &= O \left( \log n \right) 
     \label{e:Gn(o,o)} \\
  G_n(z,o;r) - G_n(o,o;r)
  &= -A(z,o;r) + O \left( |z| \frac{\log n}{n} \right)
  = O \left( \log |z| \right) 
     \label{e:Gn(z,o)} \\
  \partial^{(1)}_f G_n(z,o;r)
  &= \partial^{(1)}_f A(z,o;r)
     + O \left( \frac{\log n}{\dist(z,V_n^c)} \right) \\
  &= O \left( \frac{1}{|z|} \right) + O \left( \frac{\log n}{\dist(z,V_n^c)} \right) 
     \label{e:partial(1)} \\
  \partial^{(2)}_h G_n(z,o;r)
  &= \partial^{(2)}_h A(z,o;r) 
     + O \left( \frac{1}{n} \right) 
  = O \left( \frac{1}{|z|} \right)
     \label{e:partial(2)} \\
  \partial^{(1)}_f \partial^{(2)}_h G_n(z,o;r)
  &= \partial^{(1)}_f \partial^{(2)}_h A(z,o;r) 
     + O \left( \frac{1}{n\, \dist(z,V_n^c)} \right) \\
  &= O \left( \frac{1}{|z|^2} \right) + O \left( \frac{1}{n\, \dist(z,V_n^c)} \right).
     \label{e:partialboth} 
\end{align}
\end{lemma}

\begin{proof}[Proof of Lemma \ref{lem:estimates}]
The estimate \eqref{e:Gn(o,o)} follows from 
\eqnst
{ G_n(o,o;r) 
  \le G_n(o,o;1) 
  = \E^o [ A(S(\tau_{V_n^c})) ] = O ( \log n ). }
In order to prove \eqref{e:Gn(z,o)}, we use Lemma \ref{lem:Gn(z,w)}
to write
\eqnsplst
{ &G_n(z,o;r) - G_n(o,o;r) \\
  &\qquad = -A(z,o;r) + \E^z [ A(S_r(\tau_{V_n^c}),o;r) ] 
     - \E^o [ A(S_r(\tau_{V_n^c}),o;r) ]. } 
Due to Lemma \ref{lem:Aasymp}(ii), the first term is 
$O(\log |z|)$. By the same lemma, the random variable
inside the expectations is $O(\log n)$. Due to a difference
estimate for harmonic functions \cite{LLbook}*{Theorem 6.3.8},
the total variation distance between the exit distributions 
of the random walk (without killing) started from $z$ and $o$,
respectively, is $O(|z|/n)$. This implies the same for the
killed random walk, and the statement follows. 

The proofs of \eqref{e:partial(1)} and \eqref{e:partial(2)} are similar
to the proof of \eqref{e:partialboth}, so we only give the latter.
We write
\eqnsplst
{ \partial^{(1)}_{f} \partial^{(2)}_{h} G_n(z,o;r)
  &= - \partial^{(1)}_f \partial^{(2)}_h A(z,o;r)
    + \E^{z+f} \left[ A(S_{\bar{\tau}_n},h;r) - A(S_{\bar{\tau}_n},o;r) \right] \\
  &\qquad\qquad - \E^{z} \left[ A(S_{\bar{\tau}_n},h;r) - A(S_{\bar{\tau}_n},o;r) \right]. }
Due to Lemma \ref{lem:Aasymp}(iii), the first term is $O(|z|^{-2})$.
The random variable inside the expectations is $O(\frac{1}{n})$, again
due to Lemma \ref{lem:Aasymp}(iii). Again due to the difference estimate for
harmonic functions \cite{LLbook}*{Theorem 6.3.8}, the total variation distance
between exit distributions of the random walk (without killing) started 
from $z$ and $z+f$, respectively, is $O(1/\dist(z,V_n^c))$. This 
implies the claim.
\end{proof}

\subsubsection{Proof of the summation formulas}
\label{sssec:summation}

\begin{proof}[Proof of Proposition \ref{prop:limn}.]
In the first row of $M_L - M_\Gamma$ we can take the limit $n \to \infty$ directly
and we obtain the first row of $M_r$. In order to deal with the divergent entries,
we use row and column operations to exhibit cancellations in the determinant
that allow us to take the limit $n \to \infty$. 
Subtracting the second column from the other columns and then subtracting 
the second row from the third and fourth rows we have:
\eqnsplst
{ &\det(M_{L} - M_{\Gamma})
  = \\
  &\begin{vmatrix}
    \frac{2}{4} + o(1)                     & \frac{1}{4} + o(1)     & \frac{1}{\pi} - \frac{2}{4} + o(1)      & \frac{1}{\pi} - \frac{2}{4} + o(1) \\[1ex]
    \partial^{(2)}_{e_2} G_n                  & G_n                    & \partial^{(2)}_{-e_1} G_n                  & \partial^{(2)}_{e_1} G_n \\[1ex]
    \partial^{(1)}_{e_1}\partial^{(2)}_{e_2} G_n & \partial^{(1)}_{e_1} G_n & \partial^{(1)}_{e_1}\partial^{(2)}_{-e_1} G_n & \partial^{(1)}_{e_1}\partial^{(2)}_{e_1} G_n \\[1ex]
    (\partial^{(1)}_{-e_2}-\partial^{(1)}_{e_2}) 
        & (\partial^{(1)}_{-e_2}-\partial^{(1)}_{e_2}) G_n 
        & (\partial^{(1)}_{-e_2}-\partial^{(1)}_{e_2})
        & (\partial^{(1)}_{-e_2}-\partial^{(1)}_{e_2}) \\
     \qquad \times \partial^{(2)}_{e_2} G_n
        &  
        & \qquad \times \partial^{(2)}_{-e_1} G_n 
        & \qquad \times \partial^{(2)}_{e_1} G_n \end{vmatrix}, }  
where each entry in rows 2--4 is evaluated at $(i_1,o;r=1)$. 

We split the determinant into two terms by writing $G_n(i_1,o;1)$ 
(the second entry in the second row) as
\eqnst
{ G_n(i_1,o;1)
  = G_n(o,o;1) + (G_n(i_1,o;1) - G_n(o,o;1)). }
This gives the terms:
\eqnsplst
{ \det(M_{L} - M_{\Gamma})
  = G_n(o,o;1) \, \det(\tilde{M}^o) + \det(\tilde{M}^\diff), }
where $\tilde{M}^o$ is the minor of $M_L - M_\Gamma$ obtained 
by removing the second row and second column, 
and $\tilde{M}^\diff$ is obtained by replacing the entry $G_n$ 
by $G_n(i_1,o;1) - G_n(o,o;1)$.

The estimates of Lemma \ref{lem:estimates} with $z = i_1$ show that 
\eqnst
{ \tilde{M}^o
  = \begin{pmatrix}
    \frac{2}{4} + o(1)                    & \frac{1}{\pi} - \frac{2}{4} + o(1)    & \frac{1}{\pi} - \frac{2}{4} + o(1) \\[1ex]
    O (|z|^{-2}) + & O(|z|^{-2}) + & O(|z|^{-2}) + \\
    O (n^{-1} \dist(z,V_n^c)^{-1}) & O(n^{-1} \dist(z,V_n^c)^{-1}) & O(n^{-1} \dist(z,V_n^c)\\[1ex]
    O (|z|^{-2}) + & O(|z|^{-2}) + & O(|z|^{-2}) + \\
    O (n^{-1} \dist(z,V_n^c)^{-1}) & O(n^{-1} \dist(z,V_n^c)^{-1}) & O(n^{-1} \dist(z,V_n^c)
    \end{pmatrix}. } 
This implies that 
\eqnst
{ G_n(o,o;1) \sum_{i_1 \in V_n} \det(\tilde{M}^o)
  = G_n(o,o;1) \sum_{z \in \Z^2} \det(M^o) + O \left( \frac{\log n}{n} \right), }
where 
\eqnst
{ M^o
  := \begin{pmatrix}
    \frac{2}{4}                           & \frac{1}{\pi} - \frac{2}{4}           & \frac{1}{\pi} - \frac{2}{4} \\[1ex]
    \partial^{(1)}_{e_1}\partial^{(2)}_{e_2} A & \partial^{(1)}_{e_1}\partial^{(2)}_{-e_1} A & \partial^{(1)}_{e_1}\partial^{(2)}_{e_1} A \\[1ex]
    (\partial^{(1)}_{-e_2}-\partial^{(1)}_{e_2}) \partial^{(2)}_{e_2} A 
        & (\partial^{(1)}_{-e_2}-\partial^{(1)}_{e_2})\partial^{(2)}_{-e_1} A 
        & (\partial^{(1)}_{-e_2}-\partial_{e_2})\partial^{(2)}_{e_1} A \end{pmatrix}, }  
with each entry evaluated at $(z,o;1)$.

\begin{lemma}
\label{lem:sum0}
We have $\sum_{z \in \Z^2} \det(M^o) = 0$.
\end{lemma}

\begin{proof}
We are going to use that $A(z,o;1) = \lim_{r \uparrow 1} A(z,o;r)$. (We note that it is not
strictly necessary here to regularize, and we could argue directly at $r = 1$. 
But this helps to avoid some delicate integrability issues, and we will
need regularization anyway when we consider $\tilde{M}^\diff$.)

Due to the uniformity in $r$ of the bounds in Lemma \ref{lem:Aasymp}(iii), 
by dominated convergence we get
\eqnst
{ \sum_{z \in \Z^2} \det(M^o)
  = \lim_{r \uparrow 1} \sum_{z \in \Z^2} \det(M^o_r), }
where $M^o_r$ is defined the same way as $M^o$, but each entry evaluated 
at $(z,o;r)$. We are going to use the 
Fourier formula (see \cite{LLbook}*{Proposition 4.2.3}):
\eqnst
{ A(z,o;r)
  = \frac{1}{8 \pi^2} \iint \frac{1 - e^{i (\alpha k + \beta l)}}{2 - r(\cos \alpha + \cos \beta)}
    \, d\beta \, d\alpha, 
    \quad 0 < r < 1, } 
where $z = (k,l)$, and both integrals are over $[-\pi,\pi]$. This implies
\eqnst
{ \partial^{(1)}_{e_1} \partial^{(2)}_{e_2} A(z,o;r)
  = \frac{1}{8 \pi^2} \iint e^{i (\alpha k + \beta l)} 
    \frac{(e^{i \alpha} - 1)(e^{-i \beta} - 1)}{2 - r(\cos \alpha + \cos \beta)}
    \, d\beta \, d\alpha, }
and similar formulas hold for the other entries in rows 2--3.
It follows that 
\eqn{e:Fourier}
{ \det(M^o_r)
  = \frac{1}{64 \pi^4} \iint\!\iint \frac{e^{i(\alpha_1 + \alpha_2)k + i(\beta_1 + \beta_2)l} 
    \det(C^o) \, d\beta_1\, d\alpha_1\, d\beta_2\, d\alpha_2}
    {(2 - r (\cos \alpha_1 + \cos \beta_1))(2 - r (\cos \alpha_2 + \cos \beta_2))}, }
where
\eqnsplst
{ \det(C^o) 
  &= \begin{vmatrix}
    \frac{2}{4}                         & \frac{1}{\pi} - \frac{2}{4}       &  \frac{1}{\pi} - \frac{2}{4} \\
    (e^{i \alpha_1} - 1)(e^{-i \beta_1} - 1)  & (e^{i \alpha_1} - 1)(e^{i \alpha_1} - 1) & (e^{i \alpha_1} - 1)(e^{-i \alpha_1} - 1) \\
    -2 i \sin \beta_2 (e^{-i \beta_2} - 1) & -2 i \sin \beta_2 (e^{i \alpha_2} - 1) & -2 i \sin \beta_2 (e^{-i \alpha_2} - 1)
    \end{vmatrix} \\
  &=  -2 i \sin \beta_2 (e^{i \alpha_1} - 1) \begin{vmatrix}
    \frac{2}{4}      & \frac{1}{\pi} - \frac{2}{4} &  \frac{1}{\pi} - \frac{2}{4} \\
    e^{-i \beta_1} - 1  & e^{i \alpha_1} - 1             & e^{-i \alpha_1} - 1 \\
    e^{-i \beta_2} - 1  & e^{i \alpha_2} - 1             & e^{-i \alpha_2} - 1
    \end{vmatrix}. } 
Since the integrand in \eqref{e:Fourier} is smooth, summation over $(k,l) = z \in \Z^2$
amounts to setting $(\alpha_2, \beta_2) = (-\alpha_1, -\beta_1)$ and keeping 
only the integrals over $\alpha_1, \beta_1$. Therefore,
\eqnsplst
{ \sum_{(k,l) \in \Z^2} \det(M^o_r)
  &= \frac{1}{16 \pi^2} \iint e^{i (\alpha_1 k + \beta_1 l)} (2 i \sin \beta_1)(e^{i \alpha_1} - 1) \\
  &\qquad\qquad\qquad \times \frac{\begin{vmatrix}
    \frac{2}{4}      & \frac{1}{\pi} - \frac{2}{4} &  \frac{1}{\pi} - \frac{2}{4} \\
    e^{-i \beta_1} - 1  & e^{i \alpha_1} - 1             & e^{-i \alpha_1} - 1 \\
    e^{i \beta_1} - 1  & e^{-i \alpha_1} - 1             & e^{i \alpha_1} - 1
    \end{vmatrix}}{(2 - 2r (\cos \alpha_1 + \cos \beta_1))^2}. }
Write the factor in front of the determinant as 
$e^{i \alpha_1} - 1 = (\cos \alpha_1 - 1) + (i \sin \alpha_1)$, and
split the intergal into a sum of two terms. Then the first term is
anti-symmetric under $\alpha_1 \leftrightarrow - \alpha_1$ (since this
exchanges the second and third columns in the determinant). The second term is
anti-symmetric under the exchange $(\alpha_1,\beta_1) \leftrightarrow (-\alpha_1,-\beta_1)$
(since this exchanges the second and third rows). Hence both 
terms contribute $0$ to the integral and this completes the proof of the lemma.
\end{proof}

We return to the proof of Proposition \ref{prop:limn}.
Applying the estimates of Lemma \ref{lem:estimates} now to 
the entries of $\tilde{M}^\diff$, we get
\eqnst
{ \sum_{i_1 \in V_n} \det(\tilde{M}^\diff)
  = \sum_{z \in \Z^2} \det(M^\diff) + O \left( \frac{\log n}{n} \right), }
where
\eqnst
{ M^\diff
  = - \begin{vmatrix}
    \frac{2}{4}                            & \frac{1}{4}         & \frac{1}{\pi} - \frac{2}{4}           & \frac{1}{\pi} - \frac{2}{4}  \\[1ex]
    \partial^{(2)}_{e_2} A                   & A                  & \partial^{(2)}_{-e_1} A                  & \partial^{(2)}_{e_1} A \\[1ex]
    \partial^{(1)}_{e_1}\partial^{(2)}_{e_2} A & \partial^{(1)}_{e_1} A & \partial^{(1)}_{e_1}\partial^{(2)}_{-e_1} A & \partial^{(1)}_{e_1}\partial^{(2)}_{e_1} A \\[1ex]
    (\partial^{(1)}_{-e_2}-\partial^{(1)}_{e_2})  
        & (\partial^{(1)}_{-e_2}-\partial^{(1)}_{e_2}) A 
        & (\partial^{(1)}_{-e_2}-\partial^{(1)}_{e_2})
        & (\partial^{(1)}_{-e_2}-\partial^{(1)}_{e_2}) \\
    \quad \times \partial^{(2)}_{e_2} A
    &
    & \quad \times \partial^{(2)}_{-e_1} A 
    & \quad \times \partial^{(2)}_{e_1} A \end{vmatrix}, } 
with each entry in rows 2--4 evaluated at $(z,o;1)$.
We use that $A(z,o;1) = \lim_{r \uparrow 1} A(z,o;r)$. 
The bounds of Lemma \ref{lem:Aasymp}(ii),(iii) and dominated convergence imply that 
\eqnst
{ \sum_{z \in \Z^2} \det(M^\diff)
  = \lim_{r \uparrow 1} \sum_{z \in \Z^2} \det(M^\diff_r), }
where $M^\diff_r$ is defined in the same way as $M^\diff$, except 
the $A$-entries are evaluated at $(z,o;r)$.

Since we now have $0 < r < 1$, we can write
\eqnsplst
{ A(z,o;r) 
  &= G(o,o;r) - G(z,o;r) \\
  \partial^{(1)}_f A(z,o;r) 
  &= - \partial^{(1)}_f G(z,o;r) \\
  \partial^{(2)}_h A(z,o;r) 
  &= - \partial^{(2)}_h G(z,o;r) \\
  \partial^{(1)}_f \partial^{(2)}_h A(z,o;r) 
  &= - \partial^{(1)}_f \partial^{(2)}_h G(z,o;r). }
Due Lemma \ref{lem:sum0}, we can drop the term $G(o,o;r)$ from the $A$ entry
in $M^\diff_r$. Now undoing the row and column operations brings the
determinant to the form \eqref{e:Mr}, as required.
\end{proof}

\begin{proof}[Proof of Proposition \ref{prop:summation}.]
We use the Fourier formula:
\eqnst
{ G_{k,l}(r)
  = \frac{1}{8 \pi^2} \iint \frac{e^{i(\alpha k + \beta l)}}{D_r(\alpha,\beta)}
    \, d\alpha\, d\beta }
with variables $(\alpha_1,\beta_1)$ in row 4, with $(\alpha_2,\beta_2)$ in row 3
and $(\alpha_3,\beta_3)$ in row 2. This writes
$\det(C_{k,l}(r))$ as a 6-fold integral. Since $0 < r < 1$, the integrand is smooth,
and therefore summation over $(k,l) \in \Z^2$ amounts to setting the Fourier variables
$\alpha_3 = - (\alpha_1 + \alpha_2)$ and $\beta_3 = -(\beta_1 + \beta_2)$. 
This yields the formula in the statement.
\end{proof}

\subsection{Further computations in 2D}
\label{ssec:corr-123}

Poghosyan, Grigorev, Priezzhev and Ruelle \cite{PGPR10} extended 
Priezzhev's calculations of the height probabilities to the 
correlation function between height $0$ and height $1, 2, 3$. 
Their result is that for some explicit non-zero constants $c_h$, $d_h$
one has
\eqnsplst
{ \nu [ \eta(o) = 0 ,\, \eta(y) = h ]
  - p(0) p(h)
  &= \frac{c_h \log |y| + d_h}{|y|^4} + O \left( \frac{1}{|y|^5} \right), \\ 
  & \qquad \text{as $|y| \to \infty$, $h = 1, 2, 3$.} }
The presence of the logarithmic term in the correlation function had been 
predicted by Piroux and Ruelle \cite{PR05}, who also predicted,
based on conformal field theory calculations, that 
\eqnsplst
{ \nu [ \eta(o) = h_1 ,\, \eta(y) = h_2 ]
  - p(h_1) p(h_2)
  &\sim c_{h_1,h_2} \, \frac{\log^2 |y|}{|y|^4} \\
  & \qquad \text{as $|y| \to \infty$, $h_1, h_2 = 1, 2, 3$.} }

\begin{open}
Show that the correlation between heights $1, 2, 3$ is of order
$(\log^2 |y|)/|y|^4$.
\end{open}

The presence of logarithmic terms was brought to light by computing
height probabilities near the boundary on the discrete upper half 
plane \cites{PR05,JPR06}. Two natural boundary conditions are: 
\begin{itemize}
\item[(i)] \textbf{open}: at boundary sites $\Delta^\op_{xx} = 4$, 
one particle leaves the system on toppling; 
\item[(ii)] \textbf{closed}: at boundary sites $\Delta^\cl_{xx} = 3$, 
no particle leaves the system on toppling.
\end{itemize}
The potential kernel on the discrete upper-half plane, with either
boundary condition, is easily expressed in terms of the potential 
kernel on $\Z^2$, and Jeng, Piroux and Ruelle use this to
extend Priezzhev's calculations to these cases. After lengthy 
computations they arrive at the following result. Let
\eqnsplst
{ p^\op(i;m) 
  &= \nu^\op [ \eta((0,m)) = i ]; \\
  p^\cl(i;m)
  &= \nu^\cl [ \eta((0,m)) = i ]. }
Then, with explicit constants $a_i, b_i$, 
\eqnsplst
{ p^\op(i;m)
  &= p(i) + \frac{1}{m^2} \left( a_i + \frac{b_i}{2} + b_i \log m \right)
    + o(m^{-2}), \\
  &\qquad \text{as $m \to \infty$, $i = 1, 2, 3$;} \\
  p^\cl(i;m)
  &= p(i) - \frac{1}{m^2} \left( a_i + b_i \log m \right) 
    + o(m^{-2}), \\
  &\qquad \text{as $m \to \infty$, $i = 1, 2, 3$.} }
Jeng, Piroux and Ruelle make the remarkable observation that,
up to terms of order $o(m^{-2})$, the probabilities of heights $2$ 
and $3$ are linear combinations of the probabilities of heights 
$0$ and $1$.  That is, with either boundary condition $* = \op, \cl$
they get:
\eqnsplst
{ &\frac{48 - 12 \pi + 5 \pi^2 - \pi^3}{2(\pi - 2)} p^*(0;m)
  + (\pi - 8) p^*(1;m) + 2 (\pi - 2) p^*(2;m) \\
  &\qquad = \frac{(\pi - 2)(\pi - 1)}{\pi} + o(m^{-2}). }
They conjecture that the same relationship between the height probabilities
will hold in all domains and with any boundary condition.

\subsection{Minimal configurations}

We close this section by giving a theorem that states the 
general type of events for which the determinantal
computations sketched in Sections \ref{ssec:height-0},
\ref{ssec:corr-0-0} and \ref{ssec:scaling-0} can be applied.
At the same time, we give an alternative formulation that 
highlights the connection with the Transfer-Current Theorem.
Let $G = (V \cup \{ s \}, E)$ be a finite connected multigraph.

\begin{definition}
Let $W \subset V$, and let $\xi$ be a particle configuration 
on $W$. We say that $\xi$ is \textbf{minimal}, if the sandpile
$\eta^*$ defined by 
\eqnst
{ \eta^*(x)
  := \begin{cases}
     \xi(x) & \text{if $x \in W$;} \\
     \eta^\max(x) & \text{if $x \in V \setminus W$;}
     \end{cases} }
is recurrent, but $\eta^* - \mathbf{1}_x \not\in \cR_G$ 
for any $x \in W$.
\end{definition}

\begin{theorem}\cites{MD91,JW12}
Let $\xi$ be minimal on $W$. There exists a set of edges
$\cE_W$ touching $W$, such that 
\eqnst
{ \nu_G [ \eta : \eta(x) = \xi(x),\, x \in W ]
  = \det ( K_G (e,f) )_{e, f \in \cE_W}. }
\end{theorem}

The reason the theorem works is that minimal sandpile events
can be expressed, via a particular version of the burning bijection,
as the absence of a fixed set of edges from the uniform spanning tree.

\section{Infinite graphs}
\label{sec:measures}

In this section we look at whether the sandpile dynamics 
via topplings can be extended to infinite graphs.
Let $G = (V,E)$ be a locally finite, connected, infinite graph. 
A sequence $V_1 \subset V_2 \subset \dots \subset V$
of finite sets of vertices such that $\cup_{n=1}^\infty V_n = V$ 
is called an \textbf{exhaustion} of $G$. We let 
$G_n = (V_n \cup \{ s \}, E_n)$ denote the wired graph
obtained from $V_n$. We write $\nu_n$ for the 
stationary distribution of the sandpile Markov chain
on $G_n$. The questions we will be interested in are:\\
(i) Does $\nu_n \Rightarrow \text{ some limit } \nu$?\\
(ii) If yes, are avalanches $\nu$-a.s.~finite, if we add
a single particle to the infinite configuration? 

A fruitful approach to the above questions turns out to
be to translate them into questions about the uniform 
spanning tree via the burning bijection.

\subsection{The wired uniform spanning forest}
\label{ssec:WSF}

The following theorem, due to \cites{Pem91,Hagg95}, says 
that the measure $\UST_{G_n}$ converges
to a unique limit, regardless of the exhaustion.
In order to formulate this as weak convergence of probability
measures, regard a spanning tree as a set of edges. Then
$\UST_{G_n}$ can be be viewed as a probability measure 
on $2^E$ (note that edges in $E_n$, including the ones 
leading to $s$, are uniquely idenitified with elements of $E$).

\begin{theorem}
There exists a measure $\mathsf{WSF}$ such that for 
any exhaustion we have $\mathsf{UST}_{G_n} \Rightarrow \mathsf{WSF}$
as $n \to \infty$, in the sense of weak convergence of 
probability measures. Under $\mathsf{WSF}$, each connected 
component is an infinite tree, almost surely.
\end{theorem}

The limit $\mathsf{WSF}$ is called the 
\textbf{wired uniform spanning forest} measure.

\begin{definition}
An infinite tree has \textbf{one end}, if any two infinite 
self-avoiding paths in the tree have a finite symmetric difference.
\end{definition}

\begin{theorem}\cites{Pem91,BLPS}
\label{thm:Pemantle}
When $G = \Z^d$, we have the following.\\
(i) If $2 \le d \le 4$ then $\mathsf{WSF}$-a.s.~there is a single tree
and it has one end.\\
(ii) If $d \ge 5$ then $\mathsf{WSF}$-a.s.~there are infinitely many 
trees and each one has one end.
\end{theorem}

Given a spanning tree $\tau$ in a finite graph with a sink $s$,
we say that vertex $y$ is a \textbf{descendant} of vertex $x$,
if $x$ lies on the unique path from $y$ to $s$ in $\tau$.
This notion extends naturally to infinite one-ended trees:
in this case $y$ is called a \textbf{descendant} of $x$, 
if $y$ lies on the unique infinite self-avoiding path 
starting at $x$.

It is usually not an easy problem to decide whether for a given
infinite graph each tree has one end $\WSF$-a.s. Nevertheless the
one end property is known for a large class of graphs beyond $\Z^d$; 
see \cites{BLPS,LMS08}. Examples of graphs where the one end property 
fails are given by the direct product of $\Z$ with any finite 
connected graph. On such graphs, $\WSF$-a.s.~there is a single tree with two 
ends. See \cite{JL07} for a study of sandpiles on these graphs.

\subsection{One end property and the sandpile model}
\label{ssec:one-end}

The usefulness of the one end property for the sandpile model
is illustrated by the following theorem.

\begin{theorem}
\label{thm:measure} \cite{JW12}
Suppose that $\mathsf{WSF}$-a.s each tree has one end. Then 
there exists a measure $\nu$ such that for any exhaustion
$\nu_n \Rightarrow \nu$. That is, for any finite $Q \subset V$
and any particle configuration $\xi$ on $Q$ we have
\eqnst
{ \lim_{n \to \infty} \nu_n \left[ \eta : \eta_Q = \xi \right]
  = \nu \left[ \eta : \eta_Q = \xi \right]. }
\end{theorem}

The main idea of the proof is a decomposition of the burning
bijection of Majumdar and Dhar into two phases. Such a 
decomposition was used in \cite{Pr94} when $Q = \{ o \}$
(see Section \ref{ssec:height-123}), and is also implicit 
in \cite{MD91}. 
Fix a finite set $Q \subset V$, and suppose that our aim is to
show that under $\nu_n$, the marginal distribution of the 
sandpile in $Q$ converges as $n \to \infty$. 

\medbreak

{\bf Burning bijection anchored at $Q$.} We split the burning 
process into two phases.

\emph{Phase I.} Follow the usual burning process with the 
restriction that no vertex of $Q$ is allowed to burn. Phase I ends 
when there are no more vertices that can be burnt this way.

\emph{Phase II.} Now follow the usual burning process
to burn the remaining vertices.

A formal definition can be given along the lines of the case
$Q = \{ o \}$ presented in Section \ref{ssec:height-123}.

\medbreak

Given $\eta \in \cR_{G_n}$, let $W_n(\eta)$ denote the 
set of vertices that are burnt in Phase II (so that 
$Q \subset W_n(\eta) \subset V_n$). A bijection 
$\varphi_Q : \cR_{G_n} \to \cT_{G_n}$ can be constructed 
from the above burning rule similarly to Sections 
\ref{ssec:burning} and \ref{ssec:height-123}.

The proof of Theorem \ref{thm:measure} is based on 
the following lemma that we state without proof; see \cite{JW12}.

\begin{lemma}
\label{lem:cond-indep}
Fix $W$ such that $Q \subset W \subset V_n$. 
Under $\nu_n$, the restrictions of sandpile to $V \setminus W$ 
and $W$, respectively, are conditionally independent, given 
the event $\{ \eta : W_n(\eta) = W \}$.
\end{lemma}

\begin{proof}[Sketch of the proof of Theorem \ref{thm:measure}.]
Due to Lemma \ref{lem:cond-indep} we can write
\eqnsplst
{ \nu_n [ \eta : \eta_Q = \xi ]
  &= \sum_{Q \subset W \subset V_n} \ \nu_n [ \eta : W_n(\eta) = W,\, \eta_Q = \xi ] \\
  &= \sum_{Q \subset W \subset V_n} \ \nu_n [ \eta: W_n(\eta) = W ] \, p_{W,\xi}, }
where the numbers $p_{W,\xi}$ do not depend on $n$. On the other hand,
the event $\{ W_n = W \}$ is spanning tree local, and hence
$\nu_n [ W_n = W ] \to q_W$, for some numbers $q_W$. This suggests
that 
\eqnst
{ \nu [ \eta_Q = \xi ] 
  = \sum_{\substack{Q \subset W \subset V \\ \text{$W$ finite}}} \ q_W \, p_{W,\xi}. }
The proof can be completed by showing that the random sets $W_n$ 
converge weakly to a limit $W_\infty$ that is a.s.~finite. 
For this, the following property is key: in the \emph{first step} of Phase II, 
no vertex in $W_n \setminus Q$ can burn. Indeed, in Phase I we have 
examined such vertices, and they were found to be not burnable.
This implies that the spanning tree will contain no edges
between $W_n \setminus Q$ and $V \setminus W_n$. In fact, 
$W_n$ equals the set of all descendants of $Q$ under the bijection;
see Exercise \ref{ex:desc}. Due to the one end hypothesis of the
theorem, the set of descendants of $Q$ is finite $\WSF$-a.s. 
From this one can conclude the convergence $W_n \Rightarrow W_\infty$,
where $W_\infty$ is the set of all descendants of $Q$ in the wired
spanning forest.
\end{proof}

The following theorem shows that on transient graphs at least,
the \textbf{sandpile measure} $\nu$ constructed in 
Theorem \ref{thm:measure} is nicely behaved, in that it
has finite avalanches. The theorem can be proved along the 
lines of \cite{JR08}*{Theorem 3.11}, although that proof was
stated in $\Z^d$. Cases of graphs more general than $\Z^d$ but 
with more restrictive assumptions than those of 
Theorem \ref{thm:measure} were considered in \cite{J12}.

\begin{theorem}\cite{JR08}*{Theorem 3.11}
\label{thm:finite}
Assume the hypotheses of Theorem \ref{thm:measure}. If in addition
$G$ is transient, then for $\nu$-a.e.~$\eta$ and all $x \in V$, 
the configuration $\eta + \mathbf{1}_x$ can be stabilized with 
finitely many topplings.
\end{theorem}

\begin{proof}[Idea of the proof.]
On transient graphs, $\E_\nu [ n(x,y,\cdot) ] = G(x,y) < \infty$.
Hence $\nu$-a.s., every site topples finitely many times, when a 
particle is added at $x$. However, this is not enough, since we may
still have infinitely many vertices toppling (note that 
$\sum_{y \in V} G(x,y) = \infty$). 

In order to show that only finitely many vertices topple, one can 
use a decomposition of the avalanche into \textbf{waves}, 
introduced by Ivashkevich, Ktitarev and Priezzhev \cite{IKP94}. 
Waves are defined as follows. After 
we added a particle at $x$, topple $x$, and all other vertices 
that can be toppled, but do not allow $x$ to topple a second time.
It is not difficult to see that each vertex topples at most once
under this restriction. The set of vertices that toppled is called
the \textbf{first wave}. After the first wave, if $x$ is still
unstable (this will be the case if and only if all of its
neighbours were in the first wave), topple $x$ a second time and 
topple all other vertices that can be toppled, not allowing $x$
to topple a third time. This is called the \textbf{second wave}, etc.
Ivashkevich, Ktitarev and Priezzhev show that the ensemble of 
all possible waves started at $x$ is in bijection with the ensemble of 
all spanning forests of $G_n$ with two components such that $x$ and 
$s$ are in distinct components. 

The expected number of waves under $\nu$ is finite:
$\nu [ n(x,x,\cdot) ] = G(x,x) < \infty$; see \cite{JR08}.
Therefore, it is sufficient
to show that each wave is finite $\nu$-a.s. The latter can be deduced
from the one end assumption.
\end{proof}

On recurrent graphs, finiteness of avalanches is much more subtle,
and this is largely open.

\begin{open}
Consider $\Z^2$. Is it true that for $\nu$-a.e.~$\eta$ the 
configuration $\eta + \mathbf{1}_0$ can be stabilized 
with finitely many topplings? Note that for \emph{all} $x \in \Z^2$
we have
\eqnst
{ \E_\nu [ n(o, x; \cdot) ]
  = \lim_{n \to \infty} \E_{\nu_n} [ n(o, x; \cdot) ]
  = \lim_{n \to \infty} G_n(o,x)
  = \infty. }
Here the proof of the first equality is not immediate; see \cite{BHJ17}.
\emph{On average} every vertex topples infinitely often.
\end{open}

On graphs of the form $\Z \times G_0$, where $G_0$ is a finite connected
graph, avalanches are not finite in general. When $G_0$ consists of a single 
vertex, this follows from the fact that $\nu$ concentrates on the 
single configuration $\eta \equiv 1$ (and hence all vertices
topple infinitely often); see Exercise \ref{ex:1D} and \cite{MRSV00}. 
When $G_0$ has at least two vertices, the stationary distributions
do not have a unique weak limit point. Let $G_{-n,m}$ denote the
wired graph constructed from $\{ -n, -n+1, \dots, m-1, m \} \times G_0$, and
let $\nu_{-n,m}$ denote the stationary distribution of the sandpile 
on $G_{-n,m}$. It can be shown \cite{JL07} that there are two distinct ergodic weak limit
points of $\nu_{-n,m}$, as $n, m \to \infty$. This arises from the
fact that the burning process on $G_{-n,m}$ operates both from the left 
and the right end of the graph. A typical recurrent configuration
has a ``left-burnable'' and a ``right-burnable'' region, and these 
give rise to two distinct ergodic sandpile measures $\nu^L$ and $\nu^R$.
With respect to either $\nu^L$ and $\nu^R$, there is a strictly positive probability 
of both finite and infinite avalanches; see \cite{JL07}.

\subsection{The sandpile group of infinite graphs}

The sandpile measure $\nu$ is supported on the closed set
\eqnst
{ \cR
  = \left\{ \eta \in \prod_{x \in V} \{ 0, \dots, \deg(x)-1\} : 
    \parbox{4cm}{$\eta$ is ample for $F$ for every finite $\es \not= F \subset V$} \right\}. }
This follows from Exercise \ref{ex:comb} and weak convergence of $\nu_n$ to $\nu$. 
When avalanches are $\nu$-a.s.~finite, the addition operators
$E_x$ are defined $\nu$-a.e.~for all $x \in V$. It can be shown
that they leave $\nu$ invariant and the Abelian property holds:
$E_x E_y = E_y E_x$; see \cite{JR08}. Hence the addition operators 
generate an Abelian group of measure-preserving transformations of
$(\cR, \nu)$. [After the first version of these notes appeared
on arXiv, E.~Verbitsky (private communication) pointed 
out to us that in the case $V = \Z^d$, this group is isomorphic 
to  $\Z [ u_1^{\pm 1}, \dots, u_d^{\pm 1} ] / (f)$, where 
$(f)$ is the ideal generated by the polynomial 
$2d - \sum_{j=1}^d (u_j + u_j^{-1})$.]

The case that is perhaps best understood is sandpiles that dissipate particles
on every toppling. 
For $d \ge 1$ and an integer $\gamma \ge 1$ we define the 
\textbf{dissipative sandpile} with bulk dissipation $\gamma$ as follows.
Let $V_n \subset \Z^d$ be finite, and let 
$G^{(\gamma)}_n = (V_n \cup \{ s \}, E^{(\gamma)}_n)$ denote the 
graph obtained from the wired graph $(V_n,E_n)$ by adding $\gamma$
edges between each $x \in V_n$ and $s$. That is, a vertex $x$ in
configuration $\eta$ will be stable when $\eta(x) < 2d + \gamma$, and
when an unstable vertex is toppled, it sends $\gamma$ particles 
to the sink, in addition to sending one particle to each of its 
neighbours. The effect of toppling can be formally written 
in terms of the graph Laplacian of $G^{(\gamma)}_n$. This is
\eqnst
{ \Delta^{(\gamma)}_{xy}
  = \begin{cases}
    2d + \gamma & \text{if $x = y$;} \\
    -1          & \text{if $x \sim y$;} \\
    0           & \text{otherwise.}
    \end{cases} }
Then 
\eqnst
{ T_x \eta(y)
  = \eta(y) - \sum_{z \in V_n} \ \Delta^{(\gamma)}_{xz}, \quad x,y \in V_n. }
Maes, Redig and Saada \cite{MRS04} show that Dhar's the formalism of the sandpile group
carries through in the limit $V_n \uparrow \Z^d$, and the limiting sandpile measure
$\nu$ can be identified with Haar measure on a compact Abelian group. More precisely, 
the following is proved in \cite{MRS04}. Let 
\eqnsplst
{ \cR^{(\gamma)}
  &= \left\{ \eta \in \{ 0, \dots, 2d+\gamma-1 \}^{\Z^d} :
    \parbox{4cm}{$\eta$ is ample for $F$ for every finite $\es \not= F \subset \Z^d$} \right\} \\
  &= \left\{ \eta \in \{ 0, \dots, 2d+\gamma-1 \}^{\Z^d} :
    \parbox{4.2cm}{for every finite $\es \not= F \subset \Z^d$ there exists
    $x \in F$ such that $\eta(x) \ge \deg_F(x)$} \right\}. }
Introduce the following equivalence relation on $\cR^{(\gamma)}$. We say that 
$\eta \sim \zeta$, if there exists $m : \Z^d \to \Z$ such that 
$\eta - \zeta = \Delta^{(\gamma)} m$. Let $[\eta]$ denote the equivalence class of $\eta$.
It is shown that $\cR / \!\sim = \{ [ \eta ] : \eta \in \cR \}$ is a compact Abelian group.  
It is also shown that $\nu_n$ converges weakly to a measure $\nu$ that concentrates on 
$\cR^{(\gamma)}$, and that for $\nu$-a.e.~$\eta$ one has $[ \eta ] = \{ \eta \}$.
Finally, it is shown that $\nu$ projected to
$\cR / \sim$ is the Haar measure. 
See \cite{SV09} for interesting properties of $\cR$ and $\cR^{(\gamma)}$ related 
to the above.

\section{Stabilizability of infinite configurations}
\label{sec:infinite-conf}

In Theorem \ref{thm:finite} we saw that under certain conditions
we can add particles to a $\nu$-typical configuration, and only
finitely many topplings result a.s. A more general question that 
is interesting in its own right is: what infinite configurations can 
be stabilized (in some appropriate sense)? A more basic question 
that is still not fully understood is: what happens if we add a single column of 
$n$ particles to a stable background configuration, and attempt to 
stabilize?

\subsection{Relaxing a column of particles}
\label{ssec:column}

A recent survey that covers the topics in this section is \cite{LP17}.
If we start with a large number of particles at the origin and stabilize, 
what will be the shape of the region visited by the particles?
We collect some results on this question in three related models.
Striking computer simluations of these questions are available:
see for example \cites{LP10,LP09,PS13,LP17}.

\subsubsection{Three models}

{\bf A. Sandpile.} Start with $n$ particles at $0 \in \Z^d$, and no particles
elsewhere. Now stabilize via topplings. Let 
\eqnst
{ S_n
  = \{ x \in \Z^d : \text{$x$ was visited by a particle during stabilization} \}. }
More generally: start with $h$ particles at each $x \in \Z^d \setminus \{ 0 \}$,
where $h \le 2d - 2$ (the case $h = 2d - 1$ being trivial). Here $h$ is 
allowed to be \emph{negative}, that is, we allow a ``hole'' of depth $|h|$
that first has to be ``filled'', before topplings can occur. Let $S_{n,h}$
denote the set of vertices visited.

\medbreak

{\bf B. Rotor-router.} Start with $n$ particles at the origin and 
arbitrary initial rotors everywhere on $\Z^d$. Each particle in turn 
follows rotor-router walk until it arrives at a vertex that has not 
been visited before, and there it stops. Let 
\eqnst
{ A_n
  = \{ \text{vertices occupied after all particles stopped} \}. }

\medbreak

{\bf C. Divisible sandpile.} Start with a non-negative real mass $m$
at the origin and no mass anywhere else. If $x \in \Z^d$ has mass
$\ge 1$, distribute the mass in excess of $1$ equally among the neighbours. 
In this model topplings do not commute, but the stabilization is still
well-defined; see \cite{LP09}. Let 
\eqnst
{ D_m
  = \{ x \in \Z^d : \text{$x$ has final mass $=1$} \}. }
Heuristically, this model corresponds to 
taking $n = m |h|$ in Model A and letting $h \to -\infty$.

\subsubsection{Shape theorems / shape estimates}

Let us write $B_r = \{ x \in \Z^d : |x| < r \}$ and let 
$\omega_d = \text{volume of the unit ball in $\R^d$}$.
Levine and Peres \cites{LP09,LP17} show that the rotor-router model
and the divisible sandpile satisfy spherical shape theorems
in the strong sense that there exist $c, c' > 0$ such that 
if $m = \omega_d r^d$ then 
\eqnst
{ B_{r - c \log r} \subset A_m \subset B_{r + c' \log r}; }
and there exist $c, c' > 0$ depending on $d$ such that 
if $m = \omega_d r^d$ then 
\eqn{e:strong-circ}
{ B_{r-c} \subset D_m \subset B_{r+c'}. }

Simulations suggest that for the sandpile the asymptotic shape 
is \emph{not} circular. Levine and Peres \cite{LP09} prove 
that if $-h \ge 2-2d$, then 
\eqnst
{ B_{c_1 r - c_2} \subset S_{n,h}, }
where $c_1 = (2d-1-h)^{-1/d}$ and $c_2$ only depends on $d$.
Also, when $-h \ge 1-d$, then for every $\eps > 0$ they get 
\eqnst
{ S_{n,h} \subset B_{c'_1 r + c'_2}, }
where $c'_1 = (d - \eps - h)^{-1/d}$, and $c'_2$ depends only on
$d$, $h$ and $\eps$.
The inner and outer bounds approach each other as 
$h \downarrow -\infty$. This reinforces the idea that this 
limit corresponds to the divisible sandpile, for which the
limit shape is circular in the strong sense \eqref{e:strong-circ}. 

For the values $d \le h \le 2d - 2$, Fey, Levine and Peres \cite{FLP10} 
prove an outer bound of a cube of order $n^{1/d}$: 
for any $\eps > 0$ they get
\eqn{e:outer-cube}
{ S_n \subset \{ x \in \Z^d : \| x \|_\infty \le r \}, }
where $r = \frac{d+\eps}{2d-1-h} (n/\omega_d)^{1/d}$. 

In the analysis of all three models, an important role is played by the
\textbf{odometer} function. In the case of the sandpile model, this is
the function $v_n(x) = \text{number of topplings at $x$}$.
Let $s_n(x) = (n \mathbf{1}_o)^\circ(x)$, $x \in \Z^d$, denote the stabilization of a 
pile of $n$ particles at $o$ (with no initial holes). Comparing the number of 
incoming and outgoing particles at $x \in \Z^d$, we have
\eqnst
{ s_n(x)
  = n \mathbf{1}_o(x) - 2d v_n(x) + \sum_{y : y \sim x} v_n(y)
  = n \mathbf{1}_o(x) - (\Delta v_n)(x), }
where $\Delta$ is the discrete Laplacian on $\Z^d$. Rearranging, we have
$\Delta v_n = n \mathbf{1}_o - s_n$, where $s_n$ is between $0$ and $2d-1$,
and hence is bounded in $n$. Thus heuristically, $v_n$ should be compared 
to the function $\Phi_n$ satisfying $\Delta \Phi_n = n \mathbf{1}_o$.

\subsubsection{Scaling limit of the final configuration}

In the sandpile model, simulations show intricate fractal
patterns in the final configuration reached from a
column of height $n$. Pegden and Smart \cite{PS13} prove that this pattern
has a scaling limit. In order to state their result, 
recall that $s_n(x)$ is the stabilized
configuration:
\eqnsplst
{ s_n(x) 
  = (n \mathbf{1}_0)^\circ(x), \quad x \in \Z^d. } 
Define the rescaled function
\eqnst
{ \bar{s}_n(x)
  = s_n(n^{1/d} x), \quad x \in n^{-1/d} \Z^d, }
and extend $\bar{s}_n$ to all of $\R^d$, in such a way that it is constant 
on each cube of the form $y+[-\frac{1}{2} n^{-1/d}, \frac{1}{2} n^{-1/d})^d$,
$y \in n^{-1/d} \Z^d$.

\begin{theorem}\cite{PS13}
There exists a unique $s \in L^\infty(\R^d)$, such that 
for all functions $\varphi$ continuous with compact support
we have
\eqnst
{ \int_{\R^d} \bar{s}_n\, \varphi\, dx \stackrel{n \to \infty}{\longrightarrow}
    \int_{\R^d} s\, \varphi\, dx. }
Moreover, $\int_{\R^d} s\, dx = 1$, $0 \le s \le 2d-1$, and $s$ 
vanishes outside some ball.
\end{theorem}

\subsection{Explosions}
\label{ssec:explosions}

Given an unstable sandpile on $\Z^d$, we can attempt to stabilize it
by carrying out (legal) topplings in such a way that if at any time 
some vertex $x$ is unstable, then our procedure ensures that $x$ 
is toppled at a later time. Let us call such a toppling 
sequence \textbf{exhaustive}.

\begin{definition} We call a sandpile $\eta$ on $\Z^d$
\textbf{stabilizable}, if there exists an exhaustive toppling sequence
such that every vertex topples finitely often.
\end{definition}

\begin{definition}
A stable background configuration $\eta$ on $\Z^d$ is called 
\textbf{explosive}, if there exists $1 \le n < \infty$ such that
in attempting to stabilize $\eta + n \mathbf{1}_0$ all of 
$\Z^d$ topples. The background is called \textbf{robust}, if
there are finitely many topplings for all $n \ge 1$.
\end{definition}

{\bf Note:} Explosive implies that in fact all vertices 
topple infinitely many times.

\begin{example}
Write $\overline k$ for the configuration that equals the constant 
value $k$ everywhere. It is easy to see that $\overline{2d - 1}$ is explosive. 
On the other hand, $\overline{2d-2}$ is robust, due to \eqref{e:outer-cube}.
\end{example}

The following two examples, due to Fey, Levine and Peres \cite{FLP10},
show that there are robust configurations arbitrarily close to
$\overline{2d-1}$, and explosive ones arbitrarily close to 
$\overline{2d-2}$. For the first example, let 
\eqnst
{ \Lambda(m)
  = \{ x \in \Z^d : m \not\vert x_i,\, 1 \le i \le d \}, }
that is, remove from $\Z^d$ all vertices that have a coordinate 
divisible by $m$. Then for any $m \ge 1$ the background 
$\overline{2d-2} + \mathbf{1}_{\Lambda(m)}$ is robust;
see \cite{FLP10}*{Theorem 1.2}. For the second example, let 
\eqnst
{ \beta(x) 
  = \begin{cases}
    1 & \text{with probability $\eps$;} \\
    0 & \text{with probability $1 - \eps$.}
    \end{cases} }
Then for any $\eps > 0$, with probability $1$, 
the background $\overline{2d-2} + \beta$ is explosive; see \cite{FLP10}*{Proposition 1.4}.

\subsection{Ergodic configurations}

Finally, we state some results on stabilizability of sandpiles that are 
random samples from a translation invariant ergodic measure on 
$\{ 0, 1, 2, \dots \}^{\Z^d}$. It is tempting to assume that the
boundary for stabilizability would be given by whether the 
particle density is above or below the critical sandpile density 
$\rho_c = \E_\nu [ \eta(0) ]$. However, this is 
not so, even for product measures, as is demonstrated in various 
ways in \cite{FLW10}.

The following theorem states some results proved by
Fey and Redig \cite{FR05} and Meester and Quant \cite{MQ05}.

\begin{theorem}\cites{FR05,MQ05}
\label{thm:stabilize}
Let $\mu$ be a translation invariant ergodic measure on sandpiles on $\Z^d$. \\
(a) If $\E_\mu [ \eta(0) ] < d$, then $\mu$-a.e.~$\eta$ is stabilizable. \\
(b) If $\E_\mu [ \eta(0) ] > 2d - 1$, then $\mu$-a.e.~$\eta$ is not stabilizable.
\end{theorem}

The picture of stabilizability is more complete in dimension $d = 1$, 
since the upper and lower bounds in Theorem \ref{thm:stabilize}(a),(b) coincide.
Fey, Meester and Redig \cite{FMR09} determined what happens at the critical density.

\begin{theorem}\cite{FMR09}*{Theorem 3.2}
Let $\mu$ be a product measure on sandpiles on $\Z$ such that 
$\E_\mu [ \eta(0) ] = 1$ and $\mu [ \eta(0) = 0 ] > 0$. Then
$\mu$-a.e.~$\eta$ is not stabilizable. 
\end{theorem}

\section*{Acknowledgements}
I thank Lionel Levine and Laurent Saloff-Coste for offering the opportunity
to give a course at the Summer School. I thank Mathav Murugan for being 
available to run two tutorials at short notice. The Summer School has been a 
very stimulating environment. I thank all participants for their 
questions, comments and feedback, that I have attempted to incorporate 
into this survey. I am grateful to an anonymous referee for useful
suggestions.



\end{document}